\documentclass[reqno, 12pt]{amsart}
\usepackage[letterpaper,hmargin=1in,vmargin=1in]{geometry}
\usepackage{amsmath,amssymb,amsthm, hyperref}
\usepackage{color}

\newtheorem{theorem}[equation]{Theorem}
\newtheorem{proposition}[equation]{Proposition}
\newtheorem{corollary}[equation]{Corollary}
\newtheorem{lemma}[equation]{Lemma}
\newtheorem{definition}[equation]{Definition}

\theoremstyle{remark}
\newtheorem{remark}[equation]{Remark}

\numberwithin{equation}{section}

\newcommand{\Tr}{\operatorname{Tr}}

\newcommand{\erfc}{\operatorname{erfc}}
\newcommand{\rf}{\operatorname{rf}}
\newcommand{\lf}{\operatorname{lf}}
\newcommand{\zf}{\operatorname{zf}}
\newcommand{\tf}{\operatorname{tf}}
\newcommand{\hff}{\operatorname{hff}}
\newcommand{\hnf}{\operatorname{hnf}}
\newcommand{\hrf}{\operatorname{hrf}}
\newcommand{\hlf}{\operatorname{hlf}}
\newcommand{\td}{\operatorname{td}}
\renewcommand{\sf}{\operatorname{sf}}

\renewcommand{\epsilon}{\varepsilon}

\newcommand{\mydoi}[1]{\href{https://doi.org/#1}{doi: #1}}

\begin{document}
\title[Heat trace and Riesz means with a large Robin parameter]
{Heat trace and Riesz means with a large Robin parameter}
\author{David A. Sher}\address{Department of Mathematical Sciences, DePaul University, 2320 N. Kenmore Ave, 60614, Chicago, IL, USA.}
\email{dsher@depaul.edu}
\begin{abstract}
We consider the Laplacian with Robin boundary conditions on a smooth Riemannian manifold with boundary and investigate the behavior of various spectral quantities as the Robin parameter approaches infinity. Although the spectrum, heat trace, and Riesz means converge pointwise to their Dirichlet analogues, the asymptotics of these quantities as the spectral parameter approaches infinity, or as time approaches zero, do not. The remainders in the usual asymptotic expansions are not bounded as the Robin parameter approaches infinity. Our main result is a two-term asymptotic expansion for the Robin heat trace that depends jointly on the time variable and the Robin parameter, with a remainder that satisfies a uniform bound independent of the Robin parameter. We apply this result to obtain a similar two-term asymptotic expansion for Riesz means of order at least one, again with a remainder bound independent of the Robin parameter.
\end{abstract}

\maketitle
\thispagestyle{empty}

\section{Introduction}

Let $(\Omega,g)$ be an $m$-dimensional smooth Riemannian manifold with smooth boundary. Let $\Delta$ denote the Laplace-Beltrami operator on $\Omega$, with sign chosen so that the eigenvalues of $\Delta$ are non-negative. Let $\mathbf{n}$ denote the \emph{inward} pointing normal vector on the boundary. We consider the heat equation with Robin boundary conditions:
\begin{equation}
\begin{cases}(\partial_t + \Delta_z)u(t,z) = 0 &\quad t\in (0,\infty),\ z\in\Omega;\\
\frac{\partial u}{\partial \mathbf{n}}(t,z) = \kappa u(t,z) &\quad t\in (0,\infty),\ z\in\partial\Omega.\\
\end{cases}
\end{equation}
With our sign conventions, a positive value of $\kappa$ is the physically relevant case. 
\begin{definition} The \emph{Robin heat kernel} $H(z,z',t,\kappa)$ is the fundamental solution of the heat equation on $\Omega$ with Robin boundary conditions with parameter $\kappa$. Namely,
\begin{equation}
\begin{cases}(\partial_t + \Delta_z)H(t,z,z') = 0 &\quad t\in (0,\infty),\ z\in\Omega,\ z'\in\Omega;\\
\frac{\partial H}{\partial \mathbf{n}(z)}(t,z,z') = \kappa H(t,z,z') &\quad t\in (0,\infty),\ z\in\partial\Omega,\ z'\in\Omega;\\
\lim_{t\to 0}H(t,z,z') = \delta_{z'}(z) &\quad z\in\Omega,\ z'\in\Omega.
\end{cases}
\end{equation}
\end{definition}
\noindent We consider what happens to the Robin heat kernel and its trace as the Robin parameter $\kappa$ varies from zero to positive infinity. 

As is well-known going back to the work of McKean and Singer \cite{mckeansinger}, the trace of the heat kernel on a manifold with boundary admits a short-time asymptotic expansion in powers of $t^{1/2}$. This is true for Dirichlet, Neumann, and Robin boundary conditions. This short-time expansion in the Dirichlet and Neumann cases is
\[\Tr H(t) \sim \frac{|\Omega|}{(4\pi)^{m/2}}t^{-m/2}\pm\frac {|\partial\Omega|}{4(4\pi)^{(m-1)/2}}t^{-(m-1)/2}+O(t^{-(m-2)/2}),\]
where the sign of $\pm$ is positive for Neumann and negative for Dirichlet boundary conditions \cite[Chapter 4]{gilkey}. 

At $\kappa=0$, the Robin heat kernel is the usual Neumann heat kernel. As $\kappa$ goes to infinity, the heat kernel converges pointwise to the Dirichlet heat kernel. So one might expect that as the Robin parameter goes from zero to infinity, that the second term in the short-time heat expansion would interpolate between the values for the Neumann and Dirichlet kernels. However, it does not. As can be seen in \cite{gilkey}, the heat trace expansion for Robin boundary conditions agrees with the Neumann case up to $O(t^{-(m-2)/2})$, regardless of the value of the parameter $\kappa$. In short, \emph{the limit of the heat trace asymptotics as $\kappa\to\infty$ is not the same as the asymptotics of the limit heat trace}.

Our main theorem explains this anomaly.

\begin{theorem}\label{thm:main}
Let $\Omega^m$ be a manifold with boundary and let $t_0>0$. Consider the heat trace of the Robin Laplacian with Robin parameter $\kappa\ge 0$ on $\Omega$. Then we have for the difference of Robin and Neumann heat traces
\begin{equation}\label{eq:maintheorem}
\Tr H_{\kappa}(t) - \Tr H_0(t) = -\frac{|\partial\Omega|}{2\cdot (4\pi t)^{(m-1)/2}}(1-e^{\kappa^2t}\erfc(\kappa\sqrt t)) + R(\kappa,t),
\end{equation}
where $R(\kappa,t)\le Ct^{-(m-2)/2}$ for some $C$ depending on $\Omega$ but not on $\kappa\in [0,\infty)$ or $t\in [0,t_0)$.
\end{theorem}
\begin{remark} Without the first term on the right-hand side of \eqref{eq:maintheorem}, the remainder would fail to be uniform in $\kappa$. In other words, \eqref{eq:maintheorem} captures the part of the difference between Robin and Neumann heat kernels that is badly behaved in $\kappa$ as $\kappa\to\infty$.
\end{remark}
\begin{remark} If we fix $\kappa$ and let $t\to 0$, the correction term in \eqref{eq:maintheorem} is $O_{\kappa}(t^{(m-2)/2})$, reflecting the fact that the Robin and Neumann heat kernel asymptotics agree up to that order. However, if we let $\kappa\to\infty$ while fixing $t$, the limit of the correction term is $-|\partial\Omega|/(2\cdot (4\pi t)^{(m-1)/2})$, which is precisely the difference between the Neumann and Dirichlet heat trace asymptotics. So the asymptotics of the Robin heat trace ``at parameter infinity" agree with the asymptotics of the limiting Dirichlet heat trace, as expected.
\end{remark}
\begin{remark} There is interpolation between the Neumann and Dirichlet boundary conditions; however, it occurs as $\kappa\sqrt t$, rather than $\kappa$, goes from zero to infinity.
\end{remark}

\subsection{Relationship to prior work}

It has long been understood by physicists and mathematicians that the Robin spectrum interpolates between the Neumann and Dirichlet spectra as $\kappa$ goes to 0 and to infinity respectively. Accordingly, the variation of spectral quantities with respect to the Robin parameter has been studied carefully. In 1991, Friedlander used these ideas to give a proof of inequalities between Neumann and Dirichlet eigenvalues \cite{friedlander,filonov}. While the eigenvalues themeselves are individually well-behaved in both limits, more sophisticated spectral quantities have more unusual behavior. Ivrii showed that for any fixed $\kappa$, the second term in the asymptotics of the spectral counting functions for Robin and Neumann boundary conditions (should that term exist) are the same \cite{ivriiboundary}. However, the limit of the asymptotics as $\kappa\to\infty$ -- which is just the Neumann asymptotics -- does not agree with the asymptotics of the limiting Dirichlet counting function.

The present work is inspired by a 2012 paper of Frank and Geisinger \cite{frankgeisinger}. In this paper, the authors study Riesz means of order 1 -- a smoothed version of the eigenvalue counting function -- on domains with $C^1$ boundary. Their main idea is to couple the spectral parameter and the Robin parameter. Specifically, they look at the so-called \emph{critical scaling regime} where $\kappa\sim c\lambda^{1/2}$ for a fixed constant $c$, proving two-term asymptotics for Riesz means in that regime, as well as in the regimes where $\kappa$ grows either more slowly or more quickly than that scale. The heat analogue of the critical scaling regime is when $\kappa\sqrt t\sim c$ for a fixed constant $c$. Frank and Geisinger show that, unlike the situation for fixed $\kappa$, in this critical scaling regime, the second term in the spectral asymptotics does depend on the Robin parameter, in particular on the ratio $\kappa/\sqrt{\lambda}$.

In a series of two recent works, Baur and Larson have analyzed spectral asymptotics for cuboids in this same critical scaling regime \cite{bala1,bala2}. They first prove two-term asymptotics for Riesz means of any order in \cite{bala1} and then extend those results to the counting function itself in \cite{bala2}. Cuboids do not have $C^1$ boundary, so the work of Frank and Geisinger does not apply, but the authors are able to obtain analogous results via a connection to lattice point counting.

What we do here, in the case of smooth boundary, is provide a uniform heat trace estimate for the Robin heat kernel, valid in any regime, with a $\kappa$-independent remainder. This gives a complete and uniform characterization of the singular behavior suggested in \cite{frankgeisinger}. The critical scaling regime appears in our estimate through the first term on the right-hand side of \eqref{eq:maintheorem}, which depends precisely on the value of $\kappa\sqrt t$.

\subsection{Application to Riesz means}

We apply our heat trace estimate in order to study two-term asymptotics of Riesz means, along the lines of \cite{frankgeisinger,bala1,bala2}. This analysis is based closely on the approach and machinery developed by Frank and Larson \cite{franklarson}, who used it to, among other things, significantly sharpen the best known spectral asymptotics for convex domains. 

We use the notation of \cite{franklarson}. For $\gamma>0$, the Riesz mean of order $\gamma$ corresponding to the Robin Laplacian $\Delta^{\kappa}$ is
\[\Tr(\Delta^{\kappa}-\lambda)^{\gamma}_{-} = \sum_{\lambda_n^{\kappa}\le\lambda}(\lambda-\lambda_n^{\kappa})^{\gamma}\]
where the subscript $-$ denotes the negative part. As $\gamma$ approaches zero, these Riesz means converge pointwise to the counting function. Define the constant
\[L^{sc}_{\gamma,m}=\frac{\Gamma(\gamma+1)}{(4\pi)^{m/2}\Gamma(\gamma+\frac m2+1)}.\]
The main result of Frank and Larson, \cite[Theorem 1.1]{franklarson}, shows that if $\Omega$ has Lipschitz boundary and $\gamma>0$,
\[\Tr(\Delta^{\#}-\lambda)^{\gamma}_{-} = L^{sc}_{\gamma,m}|\Omega|\lambda^{\gamma+\frac d2}\mp\frac 14 L^{sc}_{\gamma,m-1}|\partial\Omega|\lambda^{\gamma+\frac{d-1}{2}}+o(\lambda^{\gamma+\frac{d-1}{2}}),\]
with a minus sign for Dirichlet boundary conditions and a plus sign for Neumann boundary conditions. They also give sharper estimates on the remainder, with explicit geometric constants, in the case of a convex domain. In a separate paper, Frank and Larson give two-term asymptotics for the Robin case, with fixed Robin parameter \cite{franklarson2}. Their Robin parameter may be a non-constant function on the boundary.

For our result, we consider the case when $\Omega$ has smooth boundary, $\gamma\ge 1$, and the Robin parameter is a constant $\kappa$. In that case, we give two-term asymptotics for the Riesz means. The improvement over \cite{franklarson2} is that our asymptotics are uniform in the Robin parameter $\kappa$, all the way to $\kappa=\infty$. We express our results in terms of the difference between Robin and Neumann Riesz means.
\begin{theorem}\label{thm:riesz} Suppose $\Omega$ is a smooth $m$-dimensional Riemannian manifold with smooth boundary. Suppose that $\gamma\ge 1$. Then
\begin{multline*}
\Tr (\Delta^0-\lambda)^{\gamma}_- - \Tr (\Delta^\kappa-\lambda)^{\gamma}_-
= \frac 12L^{sc}_{\gamma,m-1}|\partial\Omega|\lambda^{\gamma+(m-1)/2}\\ 
-\frac{2}{\pi}(\gamma+\frac{m-1}{2})L^{sc}_{\gamma,m-1}|\partial\Omega|\lambda^{\gamma+(m-1)/2}\int_0^{1}x\arctan(\frac{\sqrt\lambda}{\kappa}x)(1-x^2)^{\gamma+(m-3)/2}\, dx
\\+ O(\frac{\lambda^{\gamma+(m-1)/2}}{(\log|\lambda|)^{\gamma}}),
\end{multline*}
where the implied constant depends on $\gamma$ and $\Omega$ but is independent of $\kappa\in [0,\infty)$.
\end{theorem}
As an immediate corollary of Theorem \ref{thm:riesz} and \cite[Theorem 1.1]{franklarson}, we obtain two-term asymptotics for the Robin Riesz means themselves:
\begin{multline*}
\Tr (\Delta^\kappa-\lambda)^{\gamma}_-
= L^{sc}_{\gamma,m}|\Omega|\lambda^{\gamma+\frac m2} - \frac 14L^{sc}_{\gamma,m-1}|\partial\Omega|\lambda^{\gamma+(m-1)/2}\\ 
+\frac{2}{\pi}(\gamma+\frac{m-1}{2})L^{sc}_{\gamma,m-1}|\partial\Omega|\lambda^{\gamma+(m-1)/2}\int_0^{1}x\arctan(\frac{\sqrt\lambda}{\kappa}x)(1-x^2)^{\gamma+(m-3)/2}\, dx\\ + o(\lambda^{\gamma+(m-1)/2}),
\end{multline*}
again with the implied constant independent of $\kappa$.

\begin{remark} It is instructive to compare these two-term asymptotics with \cite[Theorem 1.2]{bala1}. In that paper, the authors, working on cuboids, consider $\kappa=\beta\sqrt{\lambda}$ for a fixed $\beta$ and take the $\lambda\to\infty$ asymptotics. By using the alternative expression for their coefficient $L_{\gamma,m}(\beta)$ given in \cite[Lemma 2.8]{bala1}, one sees that our two-term asymptotics are identical to theirs.

This is unsurprising, even though their boundaries are not smooth and so technically our result does not recover theirs. Heuristically, the singularities from the edges of a cuboid should contribute only at the $Ct^{-(m-2)/2}$ level to the heat trace, so one would expect the two-term Riesz mean asymptotics to be the same, as they indeed appear to be.
\end{remark}

\begin{remark} Although we only consider $\gamma\ge 1$, it seems likely that our result could be extended to the case $\gamma\in (0,1)$ by following \cite{franklarson}. We have chosen to forgo this analysis for brevity.
\end{remark}
\begin{remark} It is difficult, using the methods of this manuscript, to obtain good geometric control over the constant in the remainder estimate for \eqref{eq:maintheorem}. As such, extending the explicit geometric constant bounds for convex domains in \cite{franklarson} to the Robin case would require additional techniques.
\end{remark}

\subsection{Plan of the paper} The main technical tool in this paper is a \emph{parametrix construction} for the heat kernel, performed uniformly as the Robin parameter $\kappa$ approaches infinity. This requires specification of a good enough initial guess and a procedure to iterate away the error. The uniformity in $\kappa$ is essential in order to obtain a remainder estimate that is independent of $\kappa$. For this purpose, we use geometric microlocal analysis (or ``blow-up analysis"), following \cite{NRS25}, to build a $\kappa$-dependent calculus of operators which includes our heat kernel and in which our construction may be carried out. This type of analysis is often a productive framework for understanding noncommuting limits; see \cite{mazzeorowlett}.

 Section 2 contains geometric preliminaries, and Section 3 contains the bulk of the geometric microlocal analysis needed for our parametrix construction, including the composition theorem. In Section 4, we carefully analyze the model Robin heat kernel on a Euclidean half-space. Section 5 builds the parametrix for the heat kernel, and section 6 passes from kernel to trace and proves Theorem \ref{thm:main}. The work on Riesz means, proving Theorem \ref{thm:riesz}, is in section 7. The Appendix contains a technical proof which shows that the model Robin heat kernel is polyhomogeneous conormal on an appropriate space, a result which is needed for the work in sections 4--6.

\subsection{Acknowledgements} The author is deeply grateful to Simon Larson for extensive conversations surrounding this manuscript at all stages of its development. Thanks are also due to Matthias Baur, Rupert Frank, Iosif Polterovich, Michael Levitin, and Nikolay Filonov for helpful discussions. The author would like to thank the Isaac Newton Institute for Mathematical Sciences, Cambridge, for support and hospitality during the programme ``Geometric spectral theory and applications", where work on this project was undertaken; this work was supported by EPSRC grant EP/Z000580/1. Financial support for the author's travel was provided by AMS-Simons REGPUI grant 501949-9208.

\subsection{AI disclosure} The author occasionally used a free version of Microsoft Copilot, in temporary mode, to look up references, as well as to check and de-bug calculations. No AI was used for generating the main ideas, or any of the text, of this manuscript.


\section{Geometry and notation}

We will often work in a coordinate patch near the boundary. In such a patch, we use Riemannian normal coordinates: $z=(x,y)$, where $y$ is the distance to the boundary and $x$ is an $(m-1)$-dimensional coordinate along the boundary. This notation is chosen to match \cite{NRS25}. Without loss of generality we work in a patch contained in a tubular neighborhood of the boundary, so that $y$ is a smooth function. In these coordinates, the Riemannian metric has the form
\[g(x,y)dx^2+dy^2,\]
where $g(0,y)=1$. Assuming that $x$ is a Riemannian normal coordinate on the boundary itself, the Laplace-Beltrami operator has the following expression:
\[\Delta = \Delta_{\partial\Omega}-\partial_{yy}+yb_1(x,y)\partial_{xx}+yb_2(x,y)\partial_x+b_3(x,y)\partial_y.\]
All of the $b_i$ are smooth, as is $g$. Here $\partial_{xx}$ is shorthand for a linear combination of second derivatives in the $x$ coordinates.

In boundary normal coordinates, the heat kernel is a function
\[H(x,y,x',y',t,\kappa).\]
We also let $T=\sqrt t$ to reflect the parabolic scaling of the heat kernel. Finally, since we consider only the $\kappa\to\infty$ limit, we write $\epsilon=\kappa^{-1}$ throughout and consider $\epsilon\to 0$. We may thus view the heat kernel as a function
\[H(x,y,x',y',T,\epsilon).\]
The ``little trace" of the heat kernel is the function
\[\textrm{tr }H(z,T,\epsilon)=H(z,z,T,\epsilon),\]
and the usual heat trace is
\[\textrm{Tr }H(T,\epsilon) = \int_{\Omega}H(z,z,T,\epsilon)\, dz.\]

Finally, throughout, we use the close relationship between the Robin and Neumann heat kernels. Denote the Neumann heat kernel (the $\kappa=0$ case) by the special notation
\[H_N(z,z',T),\]
and define the ``Robin correction" by
\[H_{corr}(z,z',T,\epsilon):=H(z,z',T,\epsilon)-H_N(z,z',T).\]
For each fixed $\epsilon$, this correction is lower order in $T$. Specifically, its restriction to the diagonal $z=z'$ blows up as $T^{-1}$ rather than $T^{-2}$ as $T\to 0$. See \cite{NRS25} for more details.

\section{Heat spaces}
\subsection{Geometric microlocal preliminaries}

Throughout, we use the language of geometric microlocal analysis developed by Richard Melrose \cite{melroseaps,melrosedaomwc}. We will not include a summary here. Rather, we refer to \cite{grieser} for a general introduction to the concepts of blow-ups, polyhomogeneous conormal (or ``pc") functions, and manifolds with corners. An excellent self-contained exposition, both brief and thorough, may be found in \cite[Section 2]{gtv}; this includes b-maps, b-fibrations, and Melrose's pullback and pushforward theorems. Finally, the reader is invited to see \cite[Sections 3.1 and 3.4]{NRS25} for a presentation specifically adapted to the problem at hand.

\subsection{The trace space}

The trace space $X$ is the natural habitat for the heat trace
\[\Tr H(T,\epsilon).\]
This space is created by taking $[0,T_0)_T\times [0,\infty)_{\epsilon}$ and blowing up $T=\epsilon=0$:
\begin{equation}\label{eq:defoftracespace}
X= [[0,T_0)_T\times [0,\infty)_{\epsilon};\ T=\epsilon=0].
\end{equation}
Denote its faces by te for the lift of $T=0$, ze for the lift of $\epsilon=0$, and fe for the ``front edge" created by the blow-up at $T=\epsilon=0$.

\subsection{The single space}
The single space is where the diagonal heat kernel $H(z,z,T,\epsilon)$ lives. Initially it is a function on the space
\[M_{\epsilon} :=\Omega\times [0,\infty)_T\times [0,\infty)_{\epsilon}.\]
For each positive $\epsilon$, our space should coincide with the single heat space from \cite{NRS25}. This means blowing up $\{z\in\partial\Omega,T=0\}$. However, we also need uniformity as $\epsilon\to 0$. Inspired by \cite{hassell,hmm,ars1,ars3}, the usual way to do this is to perform all blow-ups first at $\epsilon=0$ and then for positive $\epsilon$. So we let
\begin{equation}\label{eq:singlespace}
M_{\epsilon,h}:=[[M_{\epsilon};\partial\Omega\times\{0\}\times\{0\}];\partial\Omega\times\{0\}\times [0,\infty)].
\end{equation}
We call the faces of the single space: tf for the original $t=0$ face, sf for the original lift of $\partial\Omega$, zf for the original $\epsilon=0$ face, ff$_0$ for the front face of the first blow-up (at $\epsilon=0$), and ff for the second blow-up (for all $\epsilon$).

This space is related to the trace space by the following proposition.
\begin{proposition}\label{prop:singletotrace} The map from $[M_{\epsilon,h};\{T=\epsilon=0\}]$ to $X$ given by integration in the spatial variable is a b-fibration.
\end{proposition}
\begin{proof} The map from $M_{\epsilon}$ to $[0,\infty)_T\times[0,\infty)_{\epsilon}$ is a b-fibration because it is a coordinate map. By [Lemma 2.5]\cite{hmm} it lifts to a b-fibration from
\[[M_{\epsilon};\{T=\epsilon=0\}]\to X.\]
Composing with the blow-down map yields a b-submersion from
\[[M_{\epsilon};\{T=\epsilon=0\};\ \partial\Omega\times\{0\}\times\{0\}]\to X,\]
and in fact this map is a b-fibration as the new front face has image equal to fe \cite[Proposition 3.8 part 4]{NRS25}. The same analysis yields a b-fibration from
\begin{equation}\label{eq:prepforcommute}
[M_{\epsilon};\{T=\epsilon=0\};\ \partial\Omega\times\{0\}\times\{0\};\ \partial\Omega\times\{0\}\times [0,\infty)]\to X,
\end{equation}
as the new front face has image te. 
 But we claim that the domain space in \eqref{eq:prepforcommute} is $[M_{\epsilon,h};\{T=\epsilon=0\}]$. Indeed, the first and second blow-ups in \eqref{eq:prepforcommute} are nested and therefore commute \cite[Lemma 2.1]{hmm}. After switching them, the (new) second and third blow-ups are disjoint and therefore commute \cite[Lemma 2.1]{hmm}. Switching them yields the claim, ending the proof.
\end{proof}

\subsection{The double space}

The double space is where the heat kernel itself lives. Initially the heat kernel $H(z,z',T,\epsilon)$ is a function on the interior of the manifold with corners
\[M^2_{\epsilon} := \Omega\times\Omega\times[0,\infty)_T\times [0,\infty)_{\epsilon}.\]
Denote the boundary hypersurfaces of this manifold with corners by zf for $\epsilon=0$, tf for $T=0$, and
\[\textrm{rf }=\partial\Omega\times\Omega\times [0,\infty)_{T}\times[0,\infty)_{\epsilon};\quad \textrm{lf }=\Omega\times\partial\Omega\times [0,\infty)_T\times[0,\infty)_{\epsilon}.\]
The notation here is chosen so that $\epsilon=0$ is the ``zero face", $T=0$ is the ``time face", and the other two faces are the ``right" and ``left" faces. Some authors switch the roles of right and left; we match \cite{NRS25}, which matched \cite{MaVe}.

For technical reasons, we must introduce scales with $y\sim y'\sim\epsilon$, or any two of those, at the boundary. Define $\hat M_{\epsilon}^2$ as follows:
\begin{multline}\hat M_{\epsilon}^2:=[M_{\epsilon}^2; \partial\Omega\times\partial\Omega\times [0,\infty)\times \{0\};\partial\Omega\times\partial\Omega\times[0,\infty)\times[0,\infty);\\
\partial\Omega\times\Omega\times [0,\infty)\times \{0\}; \Omega\times\partial\Omega\times [0,\infty)\times \{0\}].
\end{multline}
Call the new faces hff, hnf, hrf, and hlf respectively, as in \cite{NRS25} (here the ``n" is for ``nonzero", since this is a blow-up of $y=y'$ for nonzero $\epsilon$). Call the lifted diagonal $\Delta$; it is a p-submanifold of $\hat M_{\epsilon}^2$, intersecting hnf, hff, zf, and tf.

Consider the following four $p$-submanifolds of $\hat M_{\epsilon}^2$:
\[P_1=\Delta\cap\textrm{tf}\cap\textrm{hff};\ P_2=\Delta\cap\textrm{tf}\cap\textrm{zf};\ P_3=\Delta\cap\textrm{tf}\cap\textrm{hnf};\ P_4=\Delta\cap\textrm{tf},\]
as well as two others which will be useful later:
\[P_5=\textrm{tf}\cap\textrm{hff};\ P_6=\textrm{tf}\cap\textrm{hnf}.\]
Now set
\[M_{rh,\epsilon}^2:=[[\hat M^2_{\epsilon};P_1];P_3],\quad M_{h,\epsilon}^2=[[M_{rh,\epsilon}^2;P_2];P_4].\]
Blowing up $P_1$ creates a new face which we call sf$_0$, and blowing up $P_3$ creates sf. Blowing up $P_2$ creates td$_0$, and blowing up $P_4$ creates td. With this notation, each slice of $M_{\epsilon,h}^2$ obtained by fixing a positive $\epsilon$ is nearly isomorphic to $M_{h}^2$ from \cite{NRS25}, in the vertex-free setting. The labeling is very similar; the only difference is that our rf and lf correspond to $E_{10}$ and $E_{01}$ respectively from \cite{NRS25}. The only additional blow-up we have done is to create hnf, which is at the intersection $E_{10}\cap E_{01}$ in the notation of \cite{NRS25}.

\begin{remark} Since $P_3$ and $P_2$ are disjoint, we may blow them up in either order \cite[Proposition 3.8]{NRS25}. We have chosen $P_3$ first because the intermediate space $M^2_{rh,\epsilon}$ will be useful.
\end{remark}

\begin{lemma}\label{lem:doubletosingle} The lifted diagonal (closure of the interior of the diagonal) in $M_{\epsilon,h}^2$ is a p-submanifold of $M_{\epsilon,h}^2$ which is isometric to $M_{\epsilon,h}$.
\end{lemma}
\begin{proof} This lifted diagonal is a p-submanifold of $M^2_{rh,\epsilon}$ and is in fact already isometric to $M_{\epsilon,h}$ there. Its structure as a manifold with corners is not changed by the blow-ups of $P_2$ and $P_4$.
\end{proof}

Let $M^2_{rh}$ and $M^2_h$ be the reduced double heat space and the double heat space for the problem without $\epsilon$, defined in \cite[Section 3.3]{NRS25}. 
\begin{lemma}\label{lem:addtheeps} The map $\pi(z,z',T,\epsilon)=\pi(z,z',T)$ lifts to well-defined b-maps from $M^2_{rh,\epsilon}$ to $M^2_{rh}$ and from $M^2_{h,\epsilon}$ to $M^2_h$.
\end{lemma}
\begin{proof}It is immediate that this projection map is a b-map (indeed a b-fibration) from $M_{\epsilon}^2$ to $M^2$, which then lifts to a b-map from $\hat M_{\epsilon}^2$ to $M^2$. Creating $M^2_{rh}$ in the range means we must blow up the preimage of the new front face sf $\subseteq M^2_{rh}$ in $\hat M_{\epsilon}^2$; that pre-image is precisely $P_3$, which is blown up. So we have a b-map from $M^2_{rh,\epsilon}$ to $M^2_{rh}$. Similarly, from here, we must blow up the preimage of the new front face td $\subseteq M^2_h$ in $M^2_{rh,\epsilon}$; that pre-image is $P_4$, and it is also blown up, yielding the desired b-map from $M^2_{h,\epsilon}$ to $M^2_h$.
\end{proof}

\subsection{The triple space}
As is usual in this business, we will need to compose, in the sense of Duhamel's principle, operators whose kernels are pc on the double space. In order to do this, we create a triple space, with three spatial variables and two time variables. Unlike the previous spaces, this one is artificial. Our construction is closely modeled on the version in \cite{NRS25}. There are no vertices on the boundary to worry about, but we do have $\epsilon$.

First define
\begin{equation}
M^3_{\epsilon}:=\Omega\times\Omega\times\Omega\times[0,1)_{T'}\times[0,1)_{T''}\times[0,1)_{\epsilon},
\end{equation}
labeling its boundary hypersurfaces $F_{TL}$ for $T''=0$, $F_{TR}$ for $T'=0$, $F_{Z0}$ for $\epsilon=0$, and $F_{bii}$, $F_{ibi}$, and $F_{iib}$ for the first, second, and third spatial variables being on the boundary. We also define three projections $\pi_C$, $\pi_L$, and $\pi_R$ to $M^2_{\epsilon}$, as in \cite[(3.3)-(3.5)]{NRS25}:
\begin{equation}\label{eq:pic}
\pi_C(z,z',z'',T',T'',\epsilon)=(z,z'',\sqrt{(T')^2+(T'')^2},\epsilon);
\end{equation}
\begin{equation}\label{eq:pil}
\pi_L(z,z',z'',T',T'',\epsilon)=(z,z',T',\epsilon);
\end{equation}
\begin{equation}\label{eq:pir}
\pi_R(z,z',z'',T',T'',\epsilon)=(z',z'',T'',\epsilon).
\end{equation}
Then, putting technical details aside, if a family of operators $A=A_{\epsilon}$ has Schwarz kernel family $K_A$ on $M^2_{\epsilon}$ and $B=B_{\epsilon}$ has Schwarz kernel family $K_B$ on $M^2_{\epsilon}$, the Schwarz kernel of the composed family of operators is
\[K_C=(\pi_C)_*(\pi_L^*K_A\pi_R^*K_B).\]

We need to lift these projection maps to be maps onto $M^2_{rh,\epsilon}$ and (for $\pi_L$) onto $M^2_{h,\epsilon}$. This requires a series of blow-ups in the domain space. First we blow up $\mathcal O=\{T'=T''=0\}$ in $M^3_{\epsilon}$, calling the new front face $F_{\mathcal O}$. This blowup does not affect the $\epsilon$ variable, and so precisely as in \cite[Lemma 3.11]{NRS25}, the three projections $\pi_C$, $\pi_L$, and $\pi_R$ lift by continuity to $\Pi_C$, $\Pi_L$, and $\Pi_R$ from $[M^3_{\epsilon};\mathcal O]$ to $M^2_{\epsilon}$, and $\Pi_C$, $\Pi_L$, and $\Pi_R$ are $b$-fibrations.

Now we need to do some additional blow-ups to obtain a space $\hat M^3_{\epsilon}$ where $\Pi_C$, $\Pi_L$, and $\Pi_R$ are $b$-fibrations onto $\hat M^2_{\epsilon}$. This is done by a sequence of blow-ups affecting only the non-time variables in $M_\epsilon^3$ and thus commutes with the blow-up of $\mathcal O$. We blow up, for all values of the time variables,
\begin{itemize}
\item First $\partial\Omega\times\partial\Omega\times\partial\Omega\times \{\epsilon=0\}$. This is the intersection $F_{bii}\cap F_{ibi}\cap F_{iib}\cap F_{Z0}$. Call the new face $F_{bbb,b}$, where the first three subscripts denote that the variables are on the boundary and the last denotes that $\epsilon=0$.
\item Next the four now-disjoint submanifolds where we have either replaced a $\partial\Omega$ by $\Omega$ or allowed $\epsilon$ to be nonzero. Call the new faces $F_{bbi,b}$ et cetera.
\item Finally, the six now-disjoint submanifolds where we have made two such replacements. Call the new faces $F_{bii,b}$, et cetera.
\end{itemize}
This produces $\hat M_{\epsilon}^3$. This sequence of blow-ups is an example of a b-stretched product from \cite[Section 4]{melrosesinger}, for which the projections lift to b-fibrations \cite[Proposition 4.4]{melrosesinger}. The maps $\Pi_L$, $\Pi_C$, and $\Pi_R$ lift to products of a b-fibration in the time variable and b-fibrations in the spatial variables-plus-epsilon; therefore they are b-fibrations onto $\hat M_{\epsilon}^2$.

We need to lift these maps to be b-fibrations and/or b-maps onto $M_{rh,\epsilon}^2$. To do this we use \cite[Lemma 2.5]{hmm} repeatedly. For example, creating $M_{rh,\epsilon}^2$ from $M_{\epsilon}^2$ involves blowing up $P_1$ and then $P_3$. The lift of $P_1$ is the set\footnote{Instead of writing $\{x=x''\}$, one could more invariantly write $\{z=z''\in\partial\Omega\}$; however this is then redundant with $F_{bib,b}$, so we write it as is for emphasis.}
\[\Pi_C^*(P_1)=\mathcal O_{bib,b}:=F_{\mathcal O}\cap F_{bib,b}\cap \{x=x''\}.\]
The lift of $P_3$ is the lift of the set
\[\Pi_C^*(P_3)=\mathcal O_{bib,i}:=F_{\mathcal O}\cap F_{bib,i}\cap \{x=x''\}.\]
Since we blow up $P_1$ and then $P_3$, \cite[Lemma 2.5]{hmm} yields that the map $\Pi_C$ lifts to a well-defined b-fibration
\[\Pi_C:[\hat M^3_{\epsilon};\mathcal O_{bib,b};\mathcal O_{bib,i}]\to M^2_{rh,\epsilon}.\]
Call the newly created boundary faces $F_{\mathcal O,bib,b}$ and $F_{\mathcal O,bib,i}$.

We have to also get $\Pi_L$ and $\Pi_R$ to be b-maps onto $M^2_{rh,\epsilon}$. To this end, we compute
\[\Pi_L^*(P_1)=\mathcal O_{bbi,b}\cup \mathcal R_{bbi,b} := (F_{\mathcal O}\cap F_{bbi,b}\cap\{x=x'\}) \cup (F_{TR}\cap F_{bbi,b}\cap \{x=x'\}),\]
\[\Pi_L^*(P_3)=\mathcal O_{bbi,i}\cup \mathcal R_{bbi,i} := (F_{\mathcal O}\cap F_{bbi,i}\cap\{x=x'\}) \cup (F_{TR}\cap F_{bbi,i}\cap \{x=x'\}),\]
\[\Pi_R^*(P_1)=\mathcal O_{ibb,b}\cup \mathcal L_{ibb,b} := (F_{\mathcal O}\cap F_{ibb,b}\cap \{x'=x''\}) \cup (F_{TL}\cap F_{ibb,b}\cap \{x'=x''\}),\]
\[\Pi_R^*(P_1)=\mathcal O_{ibb,b}\cup \mathcal L_{ibb,b} := (F_{\mathcal O}\cap F_{ibb,i}\cap \{x'=x''\}) \cup (F_{TL}\cap F_{ibb,i}\cap \{x'=x''\}).\]
Observe that each of these is a union of two p-submanifolds. Finally, we will eventually need $\Pi_L$ to be a b-map onto $M^2_{h,\epsilon}$, so we compute
\[\Pi_L^*(P_2) = \mathcal O_{diag,0} \cup \mathcal R_{diag,0} :=(F_{\mathcal O}\cap \{z=z'\}\cap F_{Z0})\cup (F_{TR}\cap \{z=z'\}\cap F_{Z0});\]
\[\Pi_L^*(P_4) = \mathcal O_{diag}\cup \mathcal R_{diag} := (F_{\mathcal O}\cap \{z=z'\})\cup (F_{TR}\cap \{z=z'\}).\] 

In order to do all of the double-diagonal blow-ups at $F_{\mathcal O}$, we first do the triple-diagonal blow-up. Specifically, let
\[\hat M^3_{1,\epsilon}=[\hat M^3_{\epsilon};\mathcal O_{bib,b};\mathcal O_{bib,i};\mathcal O_{bbb,b};\mathcal O_{bbb,i}].\]
Here $\mathcal O_{bbb,b}=F_{\mathcal O}\cap\{x=x'=x''\}\cap F_{bbb,b}$ and $\mathcal O_{bbb,i}=F_{\mathcal O}\cap\{x=x'=x''\}\cap F_{bbb,i}$. 
This is a blow-up of $\hat M_{\epsilon}^3$. 
As in \cite{NRS25}, repeated use of \cite[Lemma 2.7]{hmm} shows that $\Pi_C$ lifts to a b-fibration from $\hat M^3_{1,\epsilon}$ to $M^2_{rh,\epsilon}$; the face mappings are $F_{\mathcal O,bib,b}\to \sf_0$, $F_{\mathcal O,bib,i}\to \sf$, $F_{\mathcal O,bbb,b}\to \sf_0$, $F_{\mathcal O,bbb,i}\to \sf$. By \cite[Lemma 3.9]{NRS25}, it is also the case that
\[\hat M^3_{1,\epsilon}=[\hat M^3_{\epsilon};\mathcal O_{bbb,b};\mathcal O_{bbb,i};\mathcal O_{bib,b};\mathcal O_{bib,i}].\]

We now construct the full triple heat space in two steps. First, we create a reduced triple heat space by performing the following blow-ups:
\[M_{rh,\epsilon}^3 = [\hat M_{1,\epsilon}^3;R_{bbi,b};R_{bbi,i};L_{ibb,b};L_{ibb,i};R_{diag,0},R_{diag}].\]
By \cite[Lemma 2.7]{hmm}, $\Pi_C$ lifts to a b-fibration from $M^3_{cen,\epsilon}\to M^2_{rh,\epsilon}$; the six new faces are $F_{R,bbi,b}$, $F_{R,bbi,i}$, $F_{L,ibb,b}$, $F_{L,ibb,i}$, $F_{R,diag,b}$, and $F_{R,diag}$, mapped into hrf, rf, hlf, lf, zf, and the interior respectively.

Finally, we create the full triple heat space
\[M_{h,\epsilon}^3 = [M_{rh,\epsilon}^3;\mathcal O_{bbi,b};\mathcal O_{bbi,i};\mathcal O_{ibb,b};\mathcal O_{ibb,i};\mathcal O_{diag,0};\mathcal O_{diag}].\]
The new faces are $F_{\mathcal O,bbi,b}$, $F_{\mathcal O,bbi,i}$, $F_{\mathcal O,ibb,b}$, $F_{\mathcal O,ibb,i}$, $F_{\mathcal O,diag,b}$, $F_{\mathcal O,diag}$.
$\Pi_C$ lifts, by composition with the blow-down map, to a well-defined b-map from $M^3_{rh,\epsilon}\to M^2_{rh,\epsilon}$, though this is no longer a b-fibration.

As in \cite[Proposition 3.13]{NRS25}, we claim that $\Pi_L$ and $\Pi_R$ also lift to well-defined b-maps from $M^3_{h,\epsilon}$ to $M^2_{h,\epsilon}$ and $M^2_{rh,\epsilon}$ respectively. Indeed, starting from $\hat M_{\epsilon}^3$, we needed to blow up
\begin{itemize}
\item For $\Pi_L$: $\mathcal O_{bbi,b}$, $\mathcal O_{bbi,i}$, $\mathcal R_{bbi,b}$, $\mathcal R_{bbi,i}$, $\mathcal R_{diag,0}$, $\mathcal R_{diag}$, $\mathcal O_{diag,0}$, and $\mathcal O_{diag}$.
\item For $\Pi_R$: $\mathcal O_{ibb,b}$, $\mathcal O_{ibb,i}$, $\mathcal L_{ibb,b}$, and $\mathcal L_{ibb,i}$.
\end{itemize}
Since all of these have been blown up in $M^3_{h,\epsilon}$, we are done. This may also be checked by computing the pullback of the boundary defining functions of each hypersurface in the image space; each is a product of boundary defining functions on $M^3_{rh,\epsilon}$ (but not, for example, boundary defining functions on $M_{rh,\epsilon}^3$).

\subsection{Combinatorics of composition}

\begin{lemma}\label{lem:combinatorics} The entries of the exponent matrices for the $b$-maps $\Pi_L:M^3_{rh,\epsilon}\to M^2_{h,\epsilon}$, $\Pi_R:M^3_{rh,\epsilon}\to M^2_{rh,\epsilon}$, and $\Pi_C:M_{cen,\epsilon}^3\to M^2_{rh,\epsilon}$ are all either zero or one, and the following are the entries which are one:

For $\Pi_L$: $(\mathcal F_{TR},\tf)$, $(\mathcal F_{Z0},\zf)$, $(\mathcal F_{bii},\rf)$, $(\mathcal F_{ibi},\lf)$, $(\mathcal F_{bbb,b},\hff)$, $(\mathcal F_{bbb,i},\hnf)$, $(\mathcal F_{bbi,b},\hff)$, $(\mathcal F_{bib,b},\hrf)$, $(\mathcal F_{ibb,b},\hlf)$, $(\mathcal F_{bii,b},\hrf)$, $(\mathcal F_{ibi,b},\hlf)$, $(\mathcal F_{iib,b},\zf)$, $(\mathcal F_{bbi,i},\hnf)$, $(\mathcal F_{bib,i},\rf)$, $(\mathcal F_{ibb,i},\lf)$, $(\mathcal F_{\mathcal O},\tf)$, $(\mathcal F_{\mathcal O,bbb,b},\sf_0)$, $(\mathcal F_{\mathcal O,bbb,i},\sf)$, $(\mathcal F_{\mathcal O,bbi,b},\sf_0)$, $(\mathcal F_{\mathcal O,bbi,i},\sf)$, $(\mathcal F_{\mathcal O,bib,b},\hrf)$, $(\mathcal F_{\mathcal O,bib,b},\tf)$, $(\mathcal F_{\mathcal O,bib,i},\rf)$, $(\mathcal F_{\mathcal O,bib,i},\tf)$, $(\mathcal F_{\mathcal O,ibb,b},\hlf)$, $(\mathcal F_{\mathcal O,ibb,b},\tf)$, $(\mathcal F_{\mathcal O,ibb,i},\lf)$, $(\mathcal F_{\mathcal O,ibb,i},\tf)$, $(\mathcal F_{\mathcal O,diag,b},\td_0)$, $(\mathcal F_{\mathcal O,diag},\td)$, $(R_{bbi,b},\sf_0)$, $(R_{bbi,i},\sf)$, $(L_{ibb,b},\hlf)$, $(L_{ibb,i},\lf)$, $(R_{diag,b},\td_0)$, $(R_{diag},\td)$,

with $F_{TL}$, $F_{iib}$ mapping to the interior.

For $\Pi_R$: $(\mathcal F_{TL},\tf)$, $(\mathcal F_{Z0},\zf)$, $(\mathcal F_{ibi},\rf)$, $(\mathcal F_{iib},\lf)$, 
$(\mathcal F_{bbb,b},\hff)$, $(\mathcal F_{bbb,i},\hnf)$, $(\mathcal F_{bbi,b},\hrf)$, $(\mathcal F_{bib,b},\hlf)$, $(\mathcal F_{ibb,b},\hff)$, $(\mathcal F_{bii,b},\zf)$, $(\mathcal F_{ibi,b},\hrf)$, $(\mathcal F_{iib,b},\hlf)$, 
$(\mathcal F_{bbi,i},\rf)$, $(\mathcal F_{bib,i},\lf)$, $(\mathcal F_{ibb,i},\hnf)$, 
$(\mathcal F_{\mathcal O},\tf)$, $(\mathcal F_{\mathcal O,bbb,b},\sf_0)$, $(\mathcal F_{\mathcal O,bbb,i},\sf)$, $(\mathcal F_{\mathcal O,bbi,b},\hrf)$, $(\mathcal F_{\mathcal O,bbi,b},\tf)$, $(\mathcal F_{\mathcal O,bbi,i},\rf)$, $(\mathcal F_{\mathcal O,bbi,i},\tf)$, $(\mathcal F_{\mathcal O,bib,b},\hlf)$, $(\mathcal F_{\mathcal O,bib,b},\tf)$, $(\mathcal F_{\mathcal O,bib,i},\lf)$, $(\mathcal F_{\mathcal O,bib,i},\tf)$, $(\mathcal F_{\mathcal O,ibb,b},\sf_0)$, $(\mathcal F_{\mathcal O,ibb,i},\sf)$, $(\mathcal F_{\mathcal O,diag,b},\tf)$, $(\mathcal F_{\mathcal O,diag,b},\zf)$, $(\mathcal F_{\mathcal O,diag},\tf)$, $(R_{bbi,b},\hrf)$, $(R_{bbi,i},\rf)$, $(L_{ibb,b},\sf_0)$, $(L_{ibb,i},\sf)$, $(R_{diag,b},\zf)$,

with $F_{TR}$, $F_{bii}$, and $R_{diag}$ mapping to the interior.

For $\Pi_C$: $(\mathcal F_{Z0}, \zf)$, $(\mathcal F_{bii}, \rf)$, $(\mathcal F_{iib}, \lf)$, 
$(\mathcal F_{bbb,b}, \hff)$, $(\mathcal F_{bbb,i},\hnf)$, $(\mathcal F_{bbi,b}, \hrf)$, $(\mathcal F_{bib,b}, \hff)$, $(\mathcal F_{ibb,b}, \hlf)$, $(\mathcal F_{bii,b},\hrf)$, $(\mathcal F_{ibi,b}, \zf)$, $(\mathcal F_{iib,b}, \hlf)$, 
$(\mathcal F_{bbi,i},\rf)$, $(\mathcal F_{bib,i},\hnf)$, $(\mathcal F_{ibb,i},\lf)$, 
$(\mathcal F_{\mathcal O}, \tf)$, $(\mathcal F_{\mathcal O,bbb,b},\sf_0)$, $(\mathcal F_{\mathcal O,bbb,i},\sf)$, $(\mathcal F_{\mathcal O,bib,b},\sf_0)$, $(\mathcal F_{\mathcal O,bib,i},\sf)$, $(R_{bbi,b},\hrf)$, $(R_{bbi,i},\rf)$, $(L_{ibb,b},\hlf)$, $(L_{ibb,i},\lf)$, $(R_{diag,b},\zf)$,

with $F_{TL}$, $F_{TR}$, $\mathcal F_{ibi}$, and $R_{diag}$ mapping to the interior.

\end{lemma}

\subsection{Densities}

We consider our heat kernels as full right densities in the time and space variables:
\[K_A(t',z,z',\epsilon)\, dt\, dz';\quad K_B(t'',z',z'',\epsilon)\, dt''\, dz''.\]
If we multiply the two and then integrate over $t=t'+t''$ and $z'$ we get
\[K_{A\circ B}(t,z,z'',\epsilon)\, dt\, dz''.\]
In order to use the pushforward theorem we need to convert to full densities, which corresponds to multiplying by the formal factor $dz\, d\epsilon$ after the multiplication. We must find the relationship between the natural metric densities and the canonical full densities on $M_{rh,\epsilon}^3$ and $M^2_{rh,\epsilon}$. Define:
\[\nu(M_{\epsilon}^2) := dz\, dz'\, dT\, d\epsilon;\quad \nu(M_{\epsilon}^3) := dz\, dz'\, dz''\, dT'\, dT''\, d\epsilon.\]

\begin{proposition} We have
\[\beta^*(\nu(M_{\epsilon}^2)) = \rho_{\sf_0}^{m+2}\rho_{\sf}^{m+1}\rho_{\hff}^2\rho_{\hnf}\rho_{\hlf}\rho_{\hrf}\nu(M^2_{rh,\epsilon});\]
\[\beta^*(\nu(M_{\epsilon}^3)) =\rho_{bbb,b}^5\rho_{bbi,b}^4\rho_{bib,b}^4\rho_{ibb,b}^4\rho_{bbb,i}^4\rho_{bii,b}^3\rho_{ibi,b}^3\rho_{iib,b}^3\rho_{bbi,i}^3\rho_{bib,i}^3\rho_{ibb,i}^3\]
\[\rho_{\mathcal O}\rho_{\mathcal O,bbb,b}^{2m+3}\rho_{\mathcal O,bbb,i}^{2m+2}\rho_{\mathcal O,bib,b}^{m+3}\rho_{\mathcal O,bib,i}^{m+2}\rho_{R,bbi,b}^{m+2}\rho_{R,bbi,i}^{m+1}\rho_{L,ibb,b}^{m+2}\rho_{L,ibb,i}^{m+1}\rho_{R,diag,0}^{m+1}\rho_{R,diag}^m\nu(M^3_{rh,\epsilon}).\]
\end{proposition}
\begin{proof} The proof follows from repeated iteration of [ARS C.5], which says that when blowing up a submanifold $F$ of a manifold with corners $W$, blow-up introduces a factor of $\rho^{\dim(W)-\dim(F)-1}$, with $\rho$ the boundary defining function for the new front face.

For $M_{\epsilon}^2$, we have $\dim M_{\epsilon}^2 = 2m+2$. The hff, hnf, hlf, and hrf blow-ups have codimensions 3, 2, 2, and 2 respectively and thus introduce $\rho_{\hff}^2\rho_{\hnf}\rho_{\hlf}\rho_{\hrf}$, $P_1$ introduces $\rho_{\sf_0}^{m+2}$ (where $m=\dim\Omega$), and $P_3$ introduces $\rho_{\sf}^{m+1}$.
For $M_{\epsilon}^3$, the dimension is $3m+3$. We may compute
\[\beta^*(\nu(M_{\epsilon}^3)) = \rho_{bbb,b}^5\rho_{bbi,b}^4\rho_{bib,b}^4\rho_{ibb,b}^4\rho_{bbb,i}^4\rho_{bii,b}^3\rho_{ibi,b}^3\rho_{iib,b}^3\rho_{bbi,i}^3\rho_{bib,i}^3\rho_{ibb,i}^3\rho_{\mathcal O}\nu(\hat M_{\epsilon}^3).\]
Then
\[\beta^*(\nu(\hat M_{\epsilon}^3)) = \rho_{\mathcal O,bbb,b}^{2m+2}\rho_{\mathcal O,bbb,i}^{2m+1}\rho_{\mathcal O,bib,b}^{m+2}\rho_{\mathcal O,bib,i}^{m+1}\nu(\hat M_{1,\epsilon}^3),\]
and
\[\beta^*(\nu(\hat M_{1,\epsilon}^3))=\rho_{R,bbi,b}^{m+2}\rho_{R,bbi,i}^{m+1}\rho_{L,ibb,b}^{m+2}\rho_{L,ibb,i}^{m+1}\rho_{R,diag,0}^{m+1}\rho_{R,diag}^m\nu(M_{rh,\epsilon}^3).\]
Using the fact that $\rho_{\mathcal O}$ lifts to a product of many defining functions on $M^3_{rh,\epsilon}$ yields the result.
\end{proof}

Recall also as in \cite{NRS25}
\[K_{A\circ B}\, dt\, dz\, dz'' = 2TK_{A\circ B}\nu(M_{\epsilon}^2);\]
\[K_A(t',z,z')K_B(t'',z',z'')\, dt\, dt''\, dz\, dz'\, dz'' = 4T'T''K_A(t',z,z)K_B(t'',z',z'')\nu(M_{\epsilon}^3).\]

\subsection{Composition}
The following is our composition theorem. Recall that $\mathcal A(X)$ is the space of polyhomogeneous conormal functions on a manifold with corners $X$. If we wish to specify an index family $\mathcal F$, we write $\mathcal A^{\mathcal F}(X)$.
\begin{theorem}\label{thm:composition}
Suppose that $A$ is a family of operators with Schwartz kernel in $\mathcal A(M_{\epsilon,h}^2)$ vanishing to infinite order at tf and with index sets
\[(A_{\rf},A_{\lf},A_{\zf},A_{\hff},A_{\hnf},A_{\hrf},A_{\hlf},A_{\sf},A_{\sf_0},A_{\td},A_{\td_0}).\]
Suppose that $B$ is a family of operators with Schwartz kernel in $\mathcal A(M_{\epsilon,h}^2)$ vanishing to infinite order at tf and with index sets 
\[(B_{\rf},B_{\lf},B_{\zf},B_{\hff},B_{\hnf},B_{\hrf},B_{\hlf},B_{\sf},B_{\sf_0},B_{\td}=\infty,B_{td_0}=\infty).\]
Suppose that 
\[A_{\lf}+B_{\rf} >-1\textrm{  and  } A_{\td}>-m-2.\]
Then the Schwartz kernel of the composition $A\circ B$ is an element of $\mathcal A^{\mathcal F}(M_{\epsilon,h}^2)$, where $\mathcal F$ has index sets
\begin{align*}
A_{\rf}\overline{\cup}(A_{\hnf} + B_{\rf}+3)\overline{\cup}(A_{\sf}+B_{\rf}+m+2)\textrm{ at }\rf,\\
B_{lf}\overline{\cup}(A_{\lf} + B_{\hnf}+3)\overline{\cup}(A_{\lf}+B_{\sf}+m+2)\textrm{ at }\lf,\\
(A_{\zf}+B_{\zf})\overline{\cup}(A_{\hlf}+B_{\hrf}+3)\overline{\cup}(A_{\td_0}+B_{\zf}+m+2) \textrm{ at }\zf,\\
(A_{\hff}+B_{\hff}+3)\overline{\cup}(A_{\hrf}+B_{\hlf}+2)\textrm{ at }\hff,\\
(A_{\hnf}+B_{\hnf}+3)\overline{\cup}(A_{\rf}+B_{\lf}+2)\textrm{ at }\hnf,\\
(A_{\hff}+B_{\hrf}+3)\overline{\cup}(A_{\hrf}+B_{\zf}+2)\overline{\cup}(A_{\sf_0}+B_{\hrf}+m+2)\textrm{ at }\hrf,\\
(A_{\hlf}+B_{\hff}+3)\overline{\cup}(A_{\zf}+B_{\hlf}+2)\overline{\cup}(A_{\hlf}+B_{\sf_0}+m+2)\textrm{ at }\hlf,\\
A_{\sf_0}+B_{\sf_0}+m+2\textrm{ at }\sf_0,\\
A_{\sf}+B_{\sf}+m+2\textrm{ at }\sf,\\
\infty \textrm{ at } \tf, \td,\textrm{ and }\td_0.\\
\end{align*}
\end{theorem}
\begin{remark} Note that this is consistent with the composition theorem of \cite{NRS25}. The result there, with no vertices, matches our setting with $m=2$ and without any of the $\epsilon=0$ faces. There is the additional wrinkle of our hnf face but this has little effect.
\end{remark}

\begin{proof}
By the pullback theorem and Lemma \ref{lem:combinatorics}, the index sets of $\Pi_L^*(K_A)$ are
\begin{align*}
A_{\rf}\textrm{ at } \mathcal F_{bii}, \mathcal F_{bib,i},\\
A_{\lf}\textrm{ at } \mathcal F_{ibi},L_{ibb,i},\mathcal F_{ibb,i},\\
A_{\zf}\textrm{ at } \mathcal F_{Z0},\mathcal F_{iib,b},\\
A_{\hff}\textrm{ at }\mathcal F_{bbb,b}, \mathcal F_{bbi,b},\\
A_{\hnf}\textrm{ at }\mathcal F_{bbb,i}, \mathcal F_{bbi,i},\\
A_{\hrf}\textrm{ at }\mathcal F_{bib,b},\mathcal F_{bii,b},\\
A_{\hlf}\textrm{ at }\mathcal F_{ibb,b},\mathcal F_{ibi,b},L_{ibb,b}\\
A_{\sf_0}\textrm{ at }\mathcal F_{O,bbb,b},\mathcal F_{O,bbi,b},R_{bbi,b},\\
A_{\sf}\textrm{ at }\mathcal F_{O,bbb,i},\mathcal F_{O,bbi,i},R_{bbi,i}\\
A_{\td_0}\textrm{ at }\mathcal F_{\mathcal O,diag,b},R_{diag,b}\\
A_{\td}\textrm{ at }\mathcal F_{\mathcal O,diag},R_{diag},\\
0\textrm{ at } F_{TL}, F_{iib},\\
\infty \textrm{ at }\mathcal F_{TR}, \mathcal F_{\mathcal O}, \mathcal F_{\mathcal O,bib,b},\mathcal F_{\mathcal O,bib,i},\mathcal F_{\mathcal O,ibb,b},\mathcal F_{O,ibb,i}.
\end{align*}
Similarly the index sets of $\Pi_R^*(K_B)$ are
\begin{align*}
B_{\rf}\textrm{ at } \mathcal F_{ibi}, R_{bbi,i}, \mathcal F_{bbi,i},\\
B_{\lf}\textrm{ at } \mathcal F_{iib}, \mathcal F_{bib,i},\\
B_{\zf}\textrm{ at } \mathcal F_{Z0}, R_{diag,b}, \mathcal F_{bii,b},\\
B_{\hff}\textrm{ at } \mathcal F_{bbb,b},\mathcal F_{ibb,b},\\
B_{\hnf}\textrm{ at } \mathcal F_{bbb,i},\mathcal F_{ibb,i},\\
B_{\hrf}\textrm{ at } \mathcal F_{bbi,b}, \mathcal F_{ibi,b}, R_{bbi,b},\\
B_{\hlf}\textrm{ at } \mathcal F_{bib,b}, \mathcal F_{iib,b},\\
B_{\sf_0}\textrm{ at } \mathcal F_{\mathcal O,bbb,b}, \mathcal F_{O,ibb,b}, L_{ibb,b},\\
B_{\sf}\textrm{ at } \mathcal F_{O,bbb,i}, \mathcal F_{O,ibb,i}, L_{ibb,i}\\
0\textrm{ at } \mathcal F_{TR}, \mathcal F_{bii}, \mathcal R_{diag}\\
\infty\textrm{ at } \mathcal F_{TL}, \mathcal F_{\mathcal O}, \mathcal F_{O,bbi,b}, \mathcal F_{O,bbi,i}, \mathcal F_{O,bib,b}, \mathcal F_{O,bib,i}, \mathcal F_{O,diag,b}, \mathcal F_{O,diag}.
\end{align*}
Therefore the product $4T'T''(\Pi_L^*K_A)(\Pi_R^*K_B)$ is pc on $M^3_{h,\epsilon}$ with index sets
\begin{align*}
A_{\zf}+B_{\zf}\textrm{ at } \mathcal F_{Z0}, A_{\rf} \textrm{ at }\mathcal F_{bii}, A_{\lf}+B_{\rf} \textrm{ at }\mathcal F_{ibi}, B_{lf} \textrm{ at } \mathcal F_{iib},\\
A_{\hff}+B_{\hff}\textrm{ at }\mathcal F_{bbb,b}, A_{\hnf}+B_{\hnf}\textrm{ at }\mathcal F_{bbb,i},\\
A_{\hff} + B_{\hrf}\textrm{ at }\mathcal F_{bbi,b}, A_{\hrf} + B_{\hlf}\textrm{ at }\mathcal F_{bib,b}, A_{\hlf} + B_{\hff}\textrm{ at }\mathcal F_{ibb,b},\\
A_{\hrf} + B_{\zf} \textrm{ at }\mathcal F_{bii,b},  A_{\hlf} + B_{\hrf}\textrm{ at }\mathcal F_{ibi,b}, A_{\zf} + B_{\hlf}\textrm{ at }\mathcal F_{iib,b},\\
A_{\hnf} + B_{\rf} \textrm{ at }\mathcal F_{bbi,i},  A_{\rf} + B_{\lf}\textrm{ at }\mathcal F_{bib,i}, A_{\lf} + B_{\hnf}\textrm{ at }\mathcal F_{ibb,i},\\
A_{\sf_0} + B_{\hrf} + 1\textrm{ at }R_{bbi,b}, A_{\sf} + B_{\rf} + 1\textrm{ at } R_{bbi,i},\\ A_{\hlf} + B_{\sf_0} + 1\textrm{ at }L_{ibb,b}, A_{\lf}+B_{\sf}+1 \textrm{ at }L_{ibb,i},\\
A_{\td_0} + B_{\zf}+1 \textrm{ at }R_{diag_0}, A_{\td}+1\textrm{ at }R_{diag},\\
A_{\sf_0}+B_{\sf_0}+2\textrm{ at }\mathcal F_{O,bbb,b}, A_{\sf}+B_{\sf}+2\textrm{ at }\mathcal F_{O,bbb,i},\\
\infty\textrm{ at } \mathcal F_{TL}, \mathcal F_{TR}, \mathcal F_{\mathcal O}, \mathcal F_{O,bbi,b}, \mathcal F_{O,bbi,i}, \mathcal F_{O,bib,b}, \mathcal F_{O,bib,i}, \mathcal F_{O,ibb,b}, \mathcal F_{O,ibb,i}, \mathcal F_{O,diag,b}, \mathcal F_{O,diag}.
\end{align*}
Since it vanishes to infinite order at all of the new front faces created when $M_{rh,\epsilon}^3$ is blown up to create $M_{h,\epsilon}^3$, it is in fact pc on $M^3_{rh,\epsilon}$, with the same index sets. Converting to densities:
\begin{align*}
A_{\zf}+B_{\zf}\textrm{ at } \mathcal F_{Z0}, A_{rf} \textrm{ at }\mathcal F_{bii}, A_{\lf}+B_{\rf} \textrm{ at }\mathcal F_{ibi}, B_{\lf} \textrm{ at } \mathcal F_{iib},\\
A_{\hff}+B_{\hff}+5\textrm{ at }\mathcal F_{bbb,b}, A_{\hnf}+B_{\hnf}+4\textrm{ at }\mathcal F_{bbb,i},\\
A_{\hff} + B_{\hrf}+4\textrm{ at }\mathcal F_{bbi,b}, A_{\hrf} + B_{\hlf}+4\textrm{ at }\mathcal F_{bib,b}, A_{\hlf} + B_{\hff}+4\textrm{ at }\mathcal F_{ibb,b},\\
A_{\hrf} + B_{\zf}+3 \textrm{ at }\mathcal F_{bii,b},  A_{\hlf} + B_{\hrf}+3\textrm{ at }\mathcal F_{ibi,b}, A_{\zf} + B_{\hlf}+3\textrm{ at }\mathcal F_{iib,b},\\
A_{\hnf} + B_{\rf}+3 \textrm{ at }\mathcal F_{bbi,i},  A_{\rf} + B_{\lf}+3\textrm{ at }\mathcal F_{bib,i}, A_{\lf} + B_{\hnf}+3\textrm{ at }\mathcal F_{ibb,i},\\
A_{\sf_0} + B_{\hrf} + m+3\textrm{ at }R_{bbi,b}, A_{\sf} + B_{\rf} + m+2\textrm{ at } R_{bbi,i},\\ A_{\hlf} + B_{\sf_0} + m+3\textrm{ at }L_{ibb,b}, A_{\lf}+B_{\sf}+m+2 \textrm{ at }L_{ibb,i},\\
A_{\td_0} + B_{\zf}+m+2 \textrm{ at }R_{diag_0}, A_{\td}+m+1\textrm{ at }R_{diag},\\
A_{\sf_0}+B_{\sf_0}+2m+5\textrm{ at }\mathcal F_{O,bbb,b}, A_{\sf}+B_{\sf}+2m+4\textrm{ at }\mathcal F_{O,bbb,i},\\
\infty\textrm{ at } \mathcal F_{TL}, \mathcal F_{TR}, \mathcal F_{\mathcal O}, \mathcal F_{O,bib,b}, \mathcal F_{O,bib,i}.
\end{align*}
Using the pushforward theorem gives, as a section of $\nu(M^2_{rh,\epsilon})$,
\begin{align*}
A_{\rf}\overline{\cup}(A_{\hnf} + B_{\rf}+3)\overline{\cup}(A_{\sf}+B_{\rf}+m+2)\textrm{ at } \rf,\\
B_{\lf}\overline{\cup}(A_{\lf} + B_{\hnf}+3)\overline{\cup}(A_{\lf}+B_{\sf}+m+2)\textrm{ at } \lf,\\
\infty \textrm{ at } \tf,\\
(A_{\zf}+B_{\zf})\overline{\cup}(A_{\hlf}+B_{\hrf}+3)\overline{\cup}(A_{\td_0}+B_{\zf}+m+2) \textrm{ at } \zf,\\
(A_{\hff}+B_{\hff}+5)\overline{\cup}(A_{\hrf}+B_{\hlf}+4)\textrm{ at } \hff,\\
(A_{\hnf}+B_{\hnf}+4)\overline{\cup}(A_{\rf}+B_{\lf}+3)\textrm{ at } \hnf,\\
(A_{\hff}+B_{\hrf}+4)\overline{\cup}(A_{\hrf}+B_{\zf}+3)\overline{\cup}(A_{\sf_0}+B_{\hrf}+m+3)\textrm{ at } \hrf,\\
(A_{\hlf}+B_{\hff}+4)\overline{\cup}(A_{\zf}+B_{\hlf}+3)\overline{\cup}(A_{\hlf}+B_{\sf_0}+m+3)\textrm{ at } \hlf,\\
A_{\sf_0}+B_{\sf_0}+2m+5\textrm{ at } \sf_0,\\
A_{\sf}+B_{\sf}+2m+4\textrm{ at } \sf.
\end{align*}
The conditions on the pushforward theorem that guarantee the convergence of the integral are, precisely as we assume,
\[A_{\lf}+B_{\rf} >-1;\quad A_{\td}+m+1 > -1.\]

Now go back to our correct bundle and also divide by $2T$:
\begin{align*}
A_{\rf}\overline{\cup}(A_{\hnf} + B_{\rf}+3)\overline{\cup}(A_{\sf}+B_{\rf}+m+2)\textrm{ at } \rf,\\
B_{\lf}\overline{\cup}(A_{\lf} + B_{\hnf}+3)\overline{\cup}(A_{\lf}+B_{\sf}+m+2)\textrm{ at } \lf,\\
\infty \textrm{ at } \tf,\\
(A_{\zf}+B_{\zf})\overline{\cup}(A_{\hlf}+B_{\hrf}+3)\overline{\cup}(A_{\td_0}+B_{\zf}+m+2) \textrm{ at } \zf,\\
(A_{\hff}+B_{\hff}+3)\overline{\cup}(A_{\hrf}+B_{\hlf}+2)\textrm{ at } \hff,\\
(A_{\hnf}+B_{\hnf}+3)\overline{\cup}(A_{\rf}+B_{\lf}+2)\textrm{ at } \hnf,\\
(A_{\hff}+B_{\hrf}+3)\overline{\cup}(A_{\hrf}+B_{\zf}+2)\overline{\cup}(A_{\sf_0}+B_{\hrf}+m+2)\textrm{ at } \hrf,\\
(A_{\hlf}+B_{\hff}+3)\overline{\cup}(A_{\zf}+B_{\hlf}+2)\overline{\cup}(A_{\hlf}+B_{\sf_0}+m+2)\textrm{ at } \hlf,\\
A_{\sf_0}+B_{\sf_0}+m+2\textrm{ at } \sf_0,\\
A_{\sf}+B_{\sf}+m+2\textrm{ at } \sf.
\end{align*}
This is what we wanted to prove.
\end{proof}


\section{Model: the half-plane}

In this section we analyze the Robin heat kernel on the half-space $\{(x,y)\in\mathbb R^{m-1}\times\mathbb R_+\}$. This kernel will serve as the model for our parametrix for the heat kernel on a general manifold with boundary.

The Robin heat kernel for a half-plane can be written explicitly \cite{BoFu,NRS25}. It follows that
\begin{equation}\label{eq:explicithalfspace}
H_{corr}(x,y,x',y',T,\epsilon) = 
- \frac{e^{(y+y')/\epsilon+T^2/\epsilon^2}}{(4\pi T^2)^{(m-1)/2}\epsilon}\exp[-\frac{|x-x'|^2}{4T^2}]\erfc(\frac{y+y'}{2T}+\frac{T}{\epsilon}).
\end{equation}

The key result of this section is the following, which is the reason we have constructed the double heat space as we have.
\begin{proposition}\label{prop:everythingispc} The Robin correction term $H_{corr}$ for a half-space is pc on $M^2_{rh,\epsilon}$.
\end{proposition}
Additionally, the Neumann heat kernel $H_N$ on a manifold with smooth boundary is pc on $M^2_h$. Though the proof of this fact in \cite[Section 4.2]{NRS25} is only in dimension 2, it goes through essentially verbatim for higher dimensions; see also \cite{grieserheat}.
By Lemma \ref{lem:addtheeps}, $H_N$ is pc on $M^2_{h,\epsilon}$, and thus we obtain:
\begin{corollary} The Robin heat kernel for a half-space is pc on $M^2_{h,\epsilon}$.
\end{corollary}

The proof of Proposition \ref{prop:everythingispc} is a careful exercise in asymptotic analysis in local coordinates and is deferred to the appendix. A key ingredient, both in that proof and later on, is to define an auxiliary function $v(\cdot)$, with argument in $[0,\infty)$, by
\begin{equation}\label{eq:defofv}
v(x):=\erfc(x)\exp(x^2).
\end{equation}
Then the correction term may be rewritten as
\begin{equation}\label{eq:halfspacewithv}
H_{corr}(t,x,y,x',y',\epsilon) = -\frac{1}{(4\pi T^2)^{(m-1)/2}\epsilon}\exp[-\frac{|x-x'|^2}{4T^2}-\frac{(y+y')^2}{4T^2}]v(\frac{y+y'}{2T}+\frac{T}{\epsilon}).
\end{equation}

The point of this rewrite is that $v(x)$ is a nice function, in that its leading order behavior at $x=0$ and $x=\infty$ is smooth and polynomial rather than exponential. The following is a classical observation \cite[(7.12.1)]{dlmf}:
\begin{lemma} The function $v(x)$ is smooth down to $x=0$ with $v(0)=1$. Moreover, at $x=\infty$, it is a smooth function of $x^{-1}$ with leading order $\pi^{-1/2}x^{-1}$.
\end{lemma}
In fact both statements can be upgraded to analyticity, but this is not needed for our purposes.

\subsection{Model operators and the heat kernel}
We examine the structure of $H_{corr}$ on $M^2_{rh,\epsilon}$, with an emphasis on its leading order behavior at sf and sf$_0$. In the interior of sf, away from all other boundary hypersurfaces, we use the following set of coordinates, see \cite[(4.4)]{NRS25}:
\begin{equation}\label{eq:coordsintsf}
X:=\frac{x-x'}{T};\ \xi:=\frac{y}{T};\ \xi':=\frac{y'}{T};\ x';\ \epsilon;\ \textrm{ and }T.
\end{equation}
In these coordinates, \eqref{eq:explicithalfspace} becomes
\begin{equation}\label{eq:hcorrinsf}
H_{corr}=
- \frac{e^{T(\xi+\xi')/\epsilon+T^2/\epsilon^2}}{(4\pi T^2)^{(m-1)/2}\epsilon}\exp[-\frac 14X^2]\erfc(\frac{\xi}{2}+\frac{\xi'}{2}+\frac{T}{\epsilon}).
\end{equation}
Since $T$ is our boundary defining function for sf, we see that $H_{corr}$ has leading order $-m+1$ at sf. Its leading order term there may be found by multiplying by $T^{m-1}$ and then setting $T=0$. This gives a model:
\begin{equation}\label{eq:hcorrmodelsf}
\mathcal H_{corr,-m+1,\sf}:=-\frac{1}{(4\pi)^{m-1}\epsilon}\exp[-\frac 14X^2]\erfc(\frac{\xi}{2}+\frac{\xi'}{2}).
\end{equation}

Now in the interior of sf$_0$, we use a similar set of coordinates:
\begin{equation}\label{eq:coordsintsf0}
X:=\frac{x-x'}{T};\ \xi:=\frac{y}{T};\ \xi':=\frac{y'}{T};\ x';\ \tilde\epsilon:=\frac{\epsilon}{T};\ \textrm{ and }T.
\end{equation}
These coordinates are not valid near the intersection sf$_0\cap$sf, as there $\tilde\epsilon$ is infinity. They are, however, valid in the interior of the intersections sf$_0\cap$zf and sf$_0\cap$rf\footnote{Note these are not valid at the triple intersection sf$_0\cap$zf$\cap$rf, because it no longer exists; zf$\cap$rf was blown up to create hrf.}, where $\tilde\epsilon$ is now a bdf for zf and $\xi$ is a bdf for rf.
In these coordinates,
\begin{equation}\label{eq:hcorrinsf0}
H_{corr}=
- \frac{e^{(\xi+\xi')/\tilde\epsilon+1/\tilde\epsilon^2}}{(4\pi T^2)^{(m-1)/2}\tilde\epsilon}\exp[-\frac 14X^2]\erfc(\frac{\xi}{2}+\frac{\xi'}{2}+\frac{1}{\tilde\epsilon}).
\end{equation}
We see that $H_{corr}$ has leading order $-m$ at sf$_0$, and its leading order term there is
\begin{equation}\label{eq:hcorrmodelsf0}
\mathcal H_{corr,-m,\sf_0}:=- \frac{e^{(\xi+\xi')/\tilde\epsilon+1/\tilde\epsilon^2}}{(4\pi)^{m-1}\tilde\epsilon}\exp[-\frac 14X^2]\erfc(\frac{\xi}{2}+\frac{\xi'}{2}+\frac{1}{\tilde\epsilon}).
\end{equation}

It can also be helpful to rewrite the equation using \eqref{eq:defofv}. Doing so yields
\begin{equation}\label{eq:hcorrmodelsf0withv}
H_{corr}=- \frac{1}{(4\pi T^2)^{(m-1)/2}\tilde\epsilon}\exp[-\frac 14X^2-\frac 14(\xi+\xi')^2]v(\frac{\xi}{2}+\frac{\xi'}{2}+\frac{1}{\tilde\epsilon}).
\end{equation}
The advantage of doing this is that it is now evident that $H_{corr}$ has leading order 0 at zf; the prefactor of $\tilde\epsilon^{-1}$ is cancelled by the $1/(\tilde\epsilon^{-1})$ coming from the $x\to\infty$ asymptotics of $v$. In fact, in these coordinates,
\[H_{corr}|_{\zf} = -\frac{1}{(4\pi T^2)^{(m-1)/2}\sqrt{\pi}}\exp[-\frac 14X^2-\frac 14(\xi+\xi')^2].\]
As is unsurprising, this is exactly the difference between the Neumann and Dirichlet heat kernels on the half-space.

We know a priori that $H_{corr}$ is a solution of the heat equation; let us examine this statement in the coordinates \eqref{eq:coordsintsf} and \eqref{eq:coordsintsf0}. We consider the heat operator
\[\mathcal L:=\partial_t+\Delta_z.\]
As usual with heat spaces, it is best to consider the operator $T^2\mathcal L$, which lifts to a b-operator away from the right face rf. The computation near sf proceeds precisely as in \cite[(4.12)]{NRS25} and we see that in the coordinates \eqref{eq:coordsintsf}
\begin{equation}\label{eq:heatopinsf}
T^2\mathcal L=\frac 12T\partial_T - \partial_{XX}-\frac 12X\partial_X-\partial_{\xi\xi}-\frac 12\xi\partial_{\xi}-\frac 12\xi'\partial_{\xi'}.
\end{equation}
Of course, applying \eqref{eq:heatopinsf} to \eqref{eq:hcorrinsf} necessarily yields zero. This can be confirmed by direct calculation. We also consider the coordinates \eqref{eq:coordsintsf0} in which we get
\begin{equation}\label{eq:heatopinsf0}
T^2\mathcal L=\frac 12T\partial_T - \partial_{XX}-\frac 12X\partial_X-\partial_{\xi\xi}-\frac 12\xi\partial_{\xi}-\frac 12\xi'\partial_{\xi'} - \frac 12\tilde\epsilon\partial_{\tilde\epsilon}.
\end{equation}
Applying \eqref{eq:heatopinsf0} to \eqref{eq:hcorrinsf0} also yields zero.

Finally, let us think about the Robin boundary condition for $H_{N}+H_{corr}$. The differential operator $\partial_y$ lifts to $T\partial_{\xi}$ in the coordinates \eqref{eq:coordsintsf0}. Computing directly with \eqref{eq:hcorrinsf0},
\[T\partial_{\xi}H_{corr}|_{\xi=0}=-\frac{e^{\xi'/\tilde\epsilon+1/\tilde\epsilon^2}}{(4\pi T^2)^{(m-1)/2}\tilde\epsilon}\exp[-\frac 14X^2]
\left(\frac{1}{\tilde\epsilon}\erfc(\frac{\xi'}{2}+\frac{1}{\tilde\epsilon}) -\frac{1}{\sqrt\pi}\exp(-(\frac{\xi'}{2}+\frac{1}{\tilde\epsilon})^2)
\right)\]
\[=\frac{1}{\tilde\epsilon}H_{corr}|_{\xi=0} + \frac{1}{(4\pi T^2)^{(m-1)/2}\sqrt\pi\tilde\epsilon}\exp[-\frac 14X^2-\frac 14(\xi')^2].\]
\[=\frac{1}{\tilde\epsilon}H_{corr}|_{\xi=0} + \frac{1}{\tilde\epsilon}H_{N}|_{\xi=0}.\]
So indeed, as expected, 
\[(T\partial_{\xi}-\tilde\epsilon^{-1})(H_N+H_{corr})|_{\xi=0} = 0.\]

\section{The parametrix construction}

We now construct our parametrix for the heat kernel, then improve it via a Neumann series argument to the true heat kernel. By using our composition formula, we will thereby demonstrate that the true heat kernel is pc on $M^2_{h,\epsilon}$. We will also be able to read off its leading order behavior at sf$_0$.

We need to define the function space(s) where our parametrices and errors will live. Each will live on the double space $M^2_{h,\epsilon}$. This space has twelve boundary hypersurfaces. To simplify matters, we assume that all of our kernels vanish to infinite order at tf and have leading order $0$ at -- that is, are continuous at -- each of the boundary hypersurfaces zf, hff, hnf, hrf, and hlf. 
\begin{definition} Suppose $\mathcal F=(F_{\td},F_{\td_0},F_{\sf},F_{\sf_0},F_{\rf}, F_{\lf})$ is an index family. Then $\mathcal A^{\mathcal F}_{h,\epsilon}$ is the space of kernels in $\mathcal A(M_h^2)$ with those index sets at $\td$, $\td_0$, $\sf$, $\sf_0$, $\rf$, and $\lf$, infinite order decay at $\tf$, and leading order 0 at $\zf$, $\hff$, $\hnf$, $\hrf$, and $\hlf$.
\end{definition}

\subsection{The initial guess}
The following is the first step in our parametrix construction.
\begin{proposition} There  exists an element $H^{(1)}$ of the form \eqref{eq:formofh1} with
\[H^{(1)}(z,z',T,\epsilon)\in \mathcal A^{-m,-m,-m,-m,0,0}_{h,\epsilon}\]
for which
\[T^2\mathcal L(H^{(1)})\in \mathcal A^{\infty,\infty,-m+2,-m+1,0,0}_{h,\epsilon}.\]
In other words, we can simultaneously solve the model problems to infinite order at $\td$ and $\td_0$, to two orders at $\sf$, and to one order at $\sf_0$. Moreover, we may choose $H^{(1)}$ to have a smooth expansion at $\rf$.
\end{proposition}
\begin{proof}
Our initial parametrix will be built in the form
\begin{equation}\label{eq:formofh1}
H^{(1)}(z,z',T,\epsilon) = H_N(z,z',T,\epsilon) + H_{corr}^{(1)}(z,z',T,\epsilon).
\end{equation}
First, observe that $H_N(z,z',T)\in\mathcal A^{-m,-m,-m,-m,0,0}_{h,\epsilon}$. Indeed, from \cite{NRS25}, we know that $H_N$ is pc on $M^2_{h}$. The observation now follows immediately from the pullback theorem and Lemma \ref{lem:addtheeps}. Moreover, since $H_N$ is in the kernel of $\mathcal L$, we actually have
\[T^2\mathcal L(H^{(1)}) = T^2\mathcal L(H^{(1)}_{corr}).\]
We claim that we can find an $H^{(1)}_{corr}$ with
\[H^{(1)}_{corr}\in\mathcal A^{\infty,\infty,-m+1,-m,0,0}\]
for which
\begin{equation}\label{eq:errcondh1corr}
T^2\mathcal L(H^{(1)}_{corr})\in\mathcal A^{\infty,\infty,-m+2,-m+1,0,0}.
\end{equation}
We do so by building this kernel on the \emph{reduced} heat space $M^2_{rh,\epsilon}$; it is then pc on $M^2_{h,\epsilon}$ by the pullback theorem. Assuming it vanishes to infinite order at tf in $M^2_{rh,\epsilon}$, and has leading orders $-m+1$ at sf and $-m$ at sf$_0$, and satisfies \eqref{eq:errcondh1corr}, this will be sufficient. In particular the pullback vanishes to infinite order at td and td$_0$.

To do this, we specify the leading orders at sf and sf$_0$ to be precisely the same as in the model case. Namely, let $H^{(1)}_{corr}$ have leading order behavior:
\begin{itemize}
\item At sf, leading order $-m+1$ with coefficient \eqref{eq:hcorrinsf}.
\item At sf$_0$, leading order $-m$ with coefficient \eqref{eq:hcorrinsf0}.
\item At tf, order $\infty$.
\item At all other boundary hypersurfaces of $M^2_{rh,\epsilon}$, leading order zero.
\end{itemize}
The crucial point is that we can do all of these at once because the model half-space heat kernel has all of these properties by Proposition \ref{prop:everythingispc}. Therefore the requirements are consistent and such an $H^{(1)}$ exists.

Of course, the heat operator $T^2\mathcal L$ is no longer precisely identical to the half-space heat operator. But it is the same to first order at sf and at sf$_0$. Specifically, at sf$_0$, a Chain Rule calculation gives
\begin{multline}\label{eq:heatopsf0notexact}
T^2\mathcal L=\frac 12T\partial_T - \partial_{XX}-\frac 12X\partial_X-\partial_{\xi\xi}-\frac 12\xi\partial_{\xi}-\frac 12\xi'\partial_{\xi'} - \frac 12\tilde\epsilon\partial_{\tilde\epsilon}\\ + T\xi b_1(TX+x',T\xi)\partial_{XX}+T^2\xi b_2(TX+x',T\xi)\partial_X+Tb_3(TX+x',T\xi)\partial_{\xi}.
\end{multline}
A similar calculation may be done at sf; this is identical to \cite{NRS25}. The point is that at sf$_0$ and sf, the difference between $T^2\mathcal L$ and $T^2\mathcal L_{halfspace}$ is not just a b-operator but actually $T$ times a b-operator.

We now apply this operator to $H^{(1)}_{corr}$. It is a b-operator except at rf, where $H^{(1)}_{corr}$ has a smooth expansion and therefore so does $T^2\mathcal L H^{(1)}_{corr}$. Thus $T^2\mathcal L H^{(1)}_{corr}\in\mathcal A(M^2_{rh,\epsilon})$. Its leading orders are $\infty$ at td and td$_0$. At sf, the term at order $-m+1$ is in the kernel of $T^2\mathcal L_{halfspace}$, and the difference between $T^2\mathcal L$ and $T^2\mathcal L_{halfspace}$ is $T$ times a b-operator. Thus, applying $T^2\mathcal L$ to this term yields something of order $-m+2$. Similarly, at sf$_0$, the result has order $-m+1$. This is exactly what we wanted to prove. \end{proof}

\subsection{Satisfying Robin boundary conditions}
The parametrix $H^{(1)}$ does not necessarily satisfy Robin boundary conditions. However, as we will see, it is close enough to doing so that we may correct away the error. This is morally equivalent to what is done in \cite{NRS25}, though we use slightly different language.

\begin{proposition} There  exists an element $H^{(2)}$ of the form \eqref{eq:formofh1}, satisfying Robin boundary conditions with parameter $\epsilon^{-1}$ in the left variable, with
\[H^{(2)}(z,z',T,\epsilon)\in \mathcal A^{-m,-m,-m,-m,0,0}_{h,\epsilon},\]
for which
\[T^2\mathcal L(H^{(2)})\in \mathcal A^{\infty,\infty,-m+2,-m+1,0,0}_{h,\epsilon}.\]
\end{proposition}

\begin{proof}
To do this, observe that $y$ is a globally defined function on a neighborhood of the boundary in $\Omega$, so we may consider the operator $\partial_y-\epsilon^{-1}$ and lift it from the left to act on functions on $M^2_{h,\epsilon}$ or $M^2_{rh,\epsilon}$. Call the resulting operator $X$. Observe that $yX$ is a sum of a b-operator (the lift of $y\partial_y$) and a multiplication operator (by $\epsilon^{-1}y$), so $yXH^{(1)}$ is pc on $M^2_{h,\epsilon}$, and thus $XH^{(1)}$ is also pc there. The leading orders of $XH^{(1)}$ are:
\begin{itemize}
\item 0 at rf, because $H^{(1)}$ is smooth there;
\item 0 at lf;
\item $-1$ at hrf, hlf, hff, hnf, and zf, because of the $\epsilon^{-1}$;
\item $-m$ at sf$_0$; the $\partial_y$ costs one order, as does the $\epsilon^{-1}$;
\item $-m+1$ at sf; the $\partial_y$ costs one order;
\item $-m$ at td;
\item $-m-1$ at td$_0$; the $\epsilon^{-1}$ costs one order;
\item $\infty$ at tf.
\end{itemize}
The restriction of $XH^{(1)}$ to rf is what we are most interested in. If it were zero then $H^{(1)}$ would satisfy the Robin boundary condition, but it is not.

What we will do is define
\[H^{(2)} := H^{(1)}-y\chi XH^{(1)},\]
where $\chi$ is a smooth cutoff function on $M_{\epsilon,h}^2$, equal to 1 on a small neighborhood of rf and supported inside a slightly larger neighborhood of rf. In particular, since rf does not intersect td, td$_0$, lf, hlf, or zf\footnote{This is actually a critical point -- because of the blow-up of hrf, rf and zf no longer intersect.}, we may assume that $\chi$ is identically zero near those boundary hypersurfaces. Thus $y\chi XH^{(1)}$ is pc on $M^2_{h,\epsilon}$ with leading orders:
\begin{itemize}
\item 1 at rf, because of the extra $y$;
\item 0 at hrf, hff, because of the $y$;
\item $-m+1$ at sf$_0$;
\item $-m+2$ at sf;
\item $\infty$ at tf, lf, zf, hlf, td, and td$_0$.
\end{itemize}
Adding this to $H^{(1)}$ to create $H^{(2)}$ does not worsen any of the leading orders of $H^{(1)}$. Applying the b-operator $T^2\mathcal L$ to this correction leaves the leading orders unchanged, except that the order 1 at rf becomes order 0. Moreover, if we apply $X$ to $H^{(2)}$, we get
\[XH^{(2)} = \partial_yH^{(1)}-  \chi XH^{(1)} - (y(\partial_y\chi) XH^1) - y\chi\partial_{yy}H^{(1)} - \epsilon^{-1}H^{(1)}+\epsilon^{-1}y\chi X H^{(1)},\]
which is zero when we restrict to rf; only the first, second, and fifth terms survive, and they cancel. Thus $H^{(2)}$ satisfies Robin boundary conditions, as desired. 
\end{proof}

As in \cite[Proposition 4.3]{NRS25}, we may also guarantee that the error vanishes to infinite order at rf:
\begin{proposition} There  exists an element $H^{(3)}$ of the form \eqref{eq:formofh1}, satisfying Robin boundary conditions with parameter $\epsilon^{-1}$ in the left variable, with
\[H^{(3)}(z,z',T,\epsilon)\in \mathcal A^{-m,-m,-m,-m,0,0}_{h,\epsilon},\]
for which
\[T^2\mathcal L(H^{(3)})\in \mathcal A^{\infty,\infty,-m+2,-m+1,\infty,0}_{h,\epsilon}.\]
\end{proposition}
\begin{proof} This is a similar argument to the previous proposition. Set
\[A_2:=\frac 12y^2\chi\mathcal LH^{(2)} = \frac 12y^2T^{-2}\chi(T^2\mathcal LH^{(2)}).\]
Then $A_2$ has leading orders 2 at rf, 2 at hrf and hff, $-m+1$ at sf$_0$, and $-m+2$ at sf. Consider $H^{(2)}-A_2$; it is pc, it still satisfies Robin boundary conditions, its leading orders are the same as those of $H^{(2)}$, and $T^2\mathcal LA_2$ has leading orders no worse than those of $T^2\mathcal LH^{(2)}$. But $(T^2\mathcal L)(H^{(2)}-A_2)$ now decays to first order at rf. This construction may be iterated, producing $A_j$ so that $H^{(1)}-\sum_{j=2}^kA_j$ vanishes to order $k-1$ at $E$. Doing a Borel summation, we let
\[H^{(3)}:=H^{(2)}-\sum_{j=2}^{\infty}A_j.\]
This is the desired kernel.
\end{proof}

\subsection{The Neumann series}
Finally, we invert the remaining error using a Neumann series, precisely as in \cite{NRS25}. Let
\[P^{(3)} := T^2\mathcal L(H^{(3)}).\]
Then, assuming the series is convergent, we will let
\[H^{(4)}:=H^{(3)}-H^{(3)}\sum_{j=1}^{\infty}(-T^{-2}P^{(3)})^j,\]
where kernel composition is taken in the sense of convolution operators. By construction, $H^{(4)}$ is a solution of the heat equation, with delta function initial condition, and satisfies Robin boundary conditions. It must therefore be the true heat kernel.

It remains to consider the convergence of the series. Observe that
\[-T^{-2}P^{(3)}\in\mathcal A_{h,\epsilon}^{\infty,\infty,-m,-m-1,\infty,0}.\]
We now use Theorem \ref{thm:composition}. A direct application of this theorem shows that
\[(-T^{-2}P^{(3)})^2\in\mathcal A_{h,\epsilon}^{\infty,\infty,-m+2,-m,\infty,0},\]
and a straightforward induction shows that in fact
\[(-T^{-2}P^{(3)})^j\in\mathcal A_{h,\epsilon}^{\infty,\infty,-m+2j-2,-m+j-2,\infty,0}.\]
The sum may therefore be taken asymptotically and yields an element of the same space as $-T^{-2}P^{(3)}$. However, for sufficiently large $j$, the kernel is in fact uniformly bounded by a constant independent of $t\in[0,1]$ and independent of the spatial variables. As in \cite{NRS25} and \cite{MaVe}, this shows that the sum is convergent. 

Finally, we apply the composition theorem to $H^{(3)}\sum_{j=1}^{\infty}(-T^{-2}P^{(3)})^j$. We see that it is an element of
\[\mathcal A_{h,\epsilon}^{\infty,\infty,-m+2,-m+1,0,0}.\]
Thus, by subtracting this from $H^{(3)}$, we show that the true heat kernel is pc:
\[H(z,z',T,\epsilon)\in \mathcal A_{h,\epsilon}^{-m,-m,-m,-m,0,0}.\]

We may also read off the leading order terms in the short-time expansions of $H(z,z',T,\epsilon)$ at sf, sf$_0$, td, and td$_0$. The first two orders in the asymptotic expansion of the true heat kernel at sf, and the first one order at sf$_0$, are the same as those of $H^{(3)}$. By construction, these are the same as those of $H^{(2)}$, and similarly the same as those of $H^{(1)}$, which is
\[H_{N}(z,z',T)+H_{corr}^{(1)}(z,z',T).\]
So the difference
\[H(z,z',T,\epsilon)-H_N(z,z,T)\]
which has order $-m+1$ at sf and order $-m$ at sf$_0$, has the same leading order terms as $H_{corr}^{(1)}$, which are in turn the same as those of the model half-space kernel $H_{corr}$ given in \eqref{eq:explicithalfspace}. 

In short we have shown:
\begin{itemize}
\item The Robin heat kernel is pc on $M^2_{h}$, with leading order $-m$ at each of sf, sf$_0$, td, and td$_0$.
\item Its leading orders are the same, up to but not including order $-m+2$ at sf and order $-m+1$ at sf$_0$, as those of
\[H_N(z,z',T) + H_{corr}(z,z',T,\epsilon),\]
where $H_{corr}(z,z',T,\epsilon)$ is given in \eqref{eq:explicithalfspace}. 
\end{itemize}


\section{From kernel to trace}
We now pass from the polyhomogeneity of the heat kernel to polyhomogeneity of the heat trace by first restricting to the diagonal and then integrating.

By Lemma \ref{lem:doubletosingle}, the diagonal is a $p$-submanifold of $M^2_{h,\epsilon}$, isomorphic to $M_{\epsilon,h}$. Thus by restricting to the diagonal, we see that the diagonal heat kernel
\[H(z,z,T,\epsilon)\]
is pc on $M_{\epsilon,h}$. It decays to infinite order at tf. Its leading orders at ff and ff$_0$ are the same, to two orders at ff and one order at ff$_0$, as those of
\[H_N(z,z,T) + H_{corr}(z,z,T,\epsilon),\]
where now the correction restricted to the diagonal is to leading order:
\begin{equation}\label{eq:diaghkhalfspace}
H_{corr}(z,z,T,\epsilon) = - \frac{e^{2y/\epsilon+T^2/\epsilon^2}}{(4\pi T^2)^{(m-1)/2}\epsilon}\erfc(\frac{y}{T}+\frac{T}{\epsilon}).
\end{equation}
Observe that this leading order behavior is independent of $x$.

Now apply Proposition \ref{prop:singletotrace}. Since $H(z,z,T,\epsilon)$ is pc on $M_{\epsilon,h}$ it is also pc after blowing up $\{T=\epsilon=0\}$; call the new front face of that blow-up sif$_0$. It is immediate from the pushforward theorem that
\[\Tr H(T,\epsilon)\]
is pc on $X$. The faces sf and td map into te and the faces sf$_0$, td$_0$, and sif$_0$ map into fe. We can say more by applying the pushforward theorem to each term separately. Applying it to $H_N$ gives the usual Neumann heat trace expansion for a manifold with boundary. But $H_{corr}$ vanishes to infinite order at td, td$_0$, and sif$_0$. So the expansion of $\Tr H_{corr}(T,\epsilon)$ at te depends only on the behavior of $H_{corr}(z,z,T,\epsilon)$ at sf, and the expansion at fe depends only on the behavior at sf$_0$. Furthermore, the expansions of
\[\Tr H(T,\epsilon)-\Tr H_N(T)\]
at te and fe are, to one order, the same as the expansions of the analogous correction term for a half-space.

We now identify these expansions by computing the trace of the correction term for a half-space. This involves integrating in $y$, as the integral in $x$ just gives a factor of $|\partial\Omega|$. Since
\[\exp[\frac{2y}{\epsilon}+\frac{T^2}{\epsilon^2}]\erfc(\frac yT+\frac T\epsilon)\le Ce^{-(y/T)^2},\]
the integral over any fixed range in $y$ will decay to infinite order in $T$, independent of $\epsilon$. Thus we may as well integrate all the way out to $y=\infty$. We are interested in the small-$T$ and small-$\epsilon$ asymptotics of the following integral:
\begin{equation}\label{eq:integral}
-(4\pi T^2)^{(-m+1)/2}\epsilon^{-1}\int_0^{\infty}\exp[\frac{2y}{\epsilon}+\frac{T^2}{\epsilon^2}]\erfc(\frac yT+\frac T\epsilon)\, dy.
\end{equation}
Write out the definition of the erfc function \cite[(7.2.2)]{dlmf}; we get
\begin{equation}\label{eq:integral2}
-(4\pi T^2)^{(-m+1)/2}\epsilon^{-1}\frac{2}{\sqrt\pi}\int_0^{\infty}\int_{\frac yT+\frac T\epsilon}^{\infty}\exp[\frac{2y}{\epsilon}+\frac{T^2}{\epsilon^2} - s^2]\, ds\, dy.
\end{equation}
Now change variables to $u=s-\frac yT-\frac T\epsilon$:
\begin{equation}\label{eq:integral3}
-(4\pi T^2)^{(-m+1)/2}\epsilon^{-1}\frac{2}{\sqrt\pi}\int_0^{\infty}\int_{0}^{\infty}\exp[-u^2-2u(\frac yT+\frac T\epsilon)-\frac{y^2}{T^2}]\, du\, dy.
\end{equation}
Change to $\xi=y/T$ and switch the order of integration:
\begin{equation}\label{eq:integral4}
-(4\pi T^2)^{(-m+1)/2}\epsilon^{-1}\frac{2T}{\sqrt\pi}\int_0^{\infty}\int_{0}^{\infty}\exp[-u^2-2u(\xi+\frac T\epsilon)-\xi^2]\, d\xi\, du.
\end{equation}
Rearranging yields
\begin{equation}\label{eq:integral5}
-(4\pi T^2)^{(-m+1)/2}\epsilon^{-1}\frac{2T}{\sqrt\pi}\int_0^{\infty}\exp[-2u\frac{T}{\epsilon}]\int_{0}^{\infty}\exp[-(\xi+u)^2]\, d\xi\, du,
\end{equation}
which can in turn be written in terms of an error function:
\begin{equation}\label{eq:integral6}
-(4\pi T^2)^{(-m+1)/2}\epsilon^{-1}T\int_0^{\infty}\exp[-2u\frac{T}{\epsilon}]\erfc u\,  du.
\end{equation}
This integral is a Laplace transform of the complementary error function which is well known. Using, for example, \cite[(12.13.99)]{grry}, we have
\begin{equation}\label{eq:integral7}
-(4\pi T^2)^{(-m+1)/2}\frac 12 (1-e^{(T/\epsilon)^2}\textrm{erfc}(T/\epsilon)).
\end{equation}
This encodes the leading behavior of the Robin correction term, up to an error which has order $t^{(-m+2)/2}$ at both te and fe. Our main theorem, Theorem \ref{thm:main}, follows.


\section{Riesz means}
We now pass from our heat trace results to Riesz means and prove Theorem \ref{thm:riesz}. This is essentially done by taking an inverse Laplace transform. Care is needed to obtain the appropriate remainder estimates via a Tauberian theorem. Throughout, we follow the program of Frank and Larson \cite{franklarson}.

Recall
\begin{equation}\label{eq:maintheoremrecall}
\Tr H_{0}(t) - \Tr H_{\kappa}(t) = \frac{|\partial\Omega|}{2\cdot (4\pi t)^{(m-1)/2}}(1-e^{\kappa^2t}\erfc(\kappa\sqrt t)) - R(\kappa,t).
\end{equation}
Consider the function
\[f(\lambda,\kappa) := \Tr (-\Delta^0-\lambda)_{-}-\Tr(-\Delta^k-\lambda)_{-}.\]
This is the difference of Riesz means between the Neumann and Robin traces. As in \cite[Section 4.1] {franklarson}, $f(\lambda,\kappa)$ is nondecreasing in $\lambda$. It is also nondecreasing in $\kappa$, and its $\kappa\to\infty$ limit is the function $f(\lambda)$ from \cite[Section 4.1]{franklarson}. Furthermore,
\[\int_0^{\infty}e^{-t\lambda}df(\lambda,\kappa)  = t^{-1}(\Tr H^{0}(t)-\Tr H^{\kappa}(t)).\]
We need to find a density $dh(\lambda,\kappa)$ for which
\[\int_0^{\infty}e^{-t\lambda}dh(\lambda,\kappa) = t^{-1} \frac{|\partial\Omega|}{2\cdot (4\pi t)^{(m-1)/2}}(1-e^{\kappa^2t}\erfc(\kappa\sqrt t).\]
This is an inverse Laplace transform. We get
\[dh(\lambda,\kappa) = \frac{|\partial\Omega|\lambda^{(m-1)/2}}{2\cdot(4\pi)^{(m-1)/2}\Gamma(\frac{m+1}{2})} - \frac{|\partial\Omega|}{2\cdot(4\pi)^{(m-1)/2}}\mathcal L^{-1}(t^{-(m+1)/2}e^{\kappa^2 t}\erfc(\kappa\sqrt t)).\]
\begin{lemma} We have
\[\mathcal L^{-1}(t^{-(m+1)/2}e^{\kappa^2 t}\erfc(\kappa\sqrt t)) = \frac{2}{\pi\Gamma(\frac{m-1}{2})}\int_0^{\lambda}(\lambda-u)^{\frac{m-3}{2}}\arctan(\frac{\sqrt u}{\kappa})\, du.\]
\end{lemma}
\begin{proof} This follows from the Laplace convolution theorem and the fact that
\[\mathcal L^{-1}(t^{-1}e^{\kappa^2 t}\erfc(\kappa\sqrt t)) = \frac{2}{\pi}\arctan(\frac{\sqrt\lambda}{\kappa}).\]
That fact, in turn, follows by taking the antiderivative of the Laplace transform pair \cite[(12.13.5)]{grry}.
\end{proof}
We therefore conclude
\begin{equation}
dh(\lambda,\kappa) = \frac{|\partial\Omega|\lambda^{(m-1)/2}}{2\cdot(4\pi)^{(m-1)/2}\Gamma(\frac{m+1}{2})} - \frac{|\partial\Omega|}{\pi\cdot(4\pi)^{(m-1)/2}\Gamma(\frac{m-1}{2})}\int_0^{\lambda}(\lambda-u)^{\frac{m-3}{2}}\arctan(\frac{\sqrt u}{\kappa})\, du.
\end{equation}
Introducing the notation $L^{sc}_{\gamma,d}$ from \cite{franklarson}[(8)],
\begin{equation}\label{eq:dh}
dh(\lambda,\kappa) = \frac{m+1}{4}L^{sc}_{1,m-1}|\partial\Omega|\lambda^{(m-1)/2}- \frac{m^2-1}{4\pi}L^{sc}_{1,m-1}|\partial\Omega|\int_0^{\lambda}(\lambda-u)^{\frac{m-3}{2}}\arctan(\frac{\sqrt u}{\kappa})\, du.
\end{equation}
We later need an upper bound for this, but since the second term on the right is negative we get trivially that
\begin{equation}\label{eq:upperbounddh}
dh(\lambda,\kappa)\le\frac{m+1}{4}L^{sc}_{1,m-1}|\partial\Omega|\lambda^{(m-1)/2}.
\end{equation}

Now, as in \cite[Section 4.2]{franklarson}, consider the function
\[g(\lambda,\kappa) := \Tr (-\Delta^0-\lambda)_- - \Tr(-\Delta^{\kappa}-\lambda)_- - h(\lambda,\kappa).\] 
We just showed that $g+h$ is nondecreasing in $\lambda$ and $\kappa$. Therefore the Borel measure
\[\int_{\omega} dg(\lambda,\kappa) + \int_{\omega}dh(\lambda,\kappa)\]
is non-negative. More to the point, by \eqref{eq:upperbounddh}, the Borel measure
\[\tilde\mu(\omega) := \int_{\omega} dg(\lambda,\kappa) + \frac{m+1}{4}L^{sc}_{1,m-1}|\partial\Omega|\int_{\omega}\lambda^{(m-1)/2}d\lambda\]
is also non-negative.
Observe additionally that
\[\int_0^{\infty}e^{-t\lambda} dg(\lambda,\kappa) = t^{-1}(\Tr H^0(t) - \Tr H^k(t)- t^{-1}\frac{|\partial\Omega|}{2\cdot (4\pi t)^{(m-1)/2}}(1-e^{\kappa^2t}\erfc(\kappa\sqrt t)).\]
By Theorem \ref{thm:main}, there is a constant $C_0$ independent of $\kappa$ for which
\begin{equation}\label{eq:hkremainderineq}
\left|\int_0^{\infty}e^{-t\lambda} dg(\lambda,\kappa)\right|=\left|t^{-1}R(\kappa,t)\right|\le C_0t^{-m/2}.
\end{equation}

We now apply a Tauberian remainder theorem \cite[Proposition 8.1]{franklarson}. In their notation, $u=\lambda$, $d=m$, $\tilde\mu$ is $\tilde\mu$, $\mu$ is $dg(\lambda,\kappa)$,  $N=1$, $\nu_1=(m+1)/2$, and $K_1=\frac 12L^{sc}_{1,m-1}|\partial\Omega|$. We also use $\gamma=\gamma-1$; this is non-negative as required because our $\gamma$ is at least 1. The result of this theorem is the following. For some constants $B$, $C$, and $k_0$ depending only on $\gamma$ and $m$, and for any $k\ge k_0$,
\begin{equation}\label{eq:tauberianresult}
\left|\int_0^{\lambda}(1-v/\lambda)^{\gamma-1}dg(v,\kappa) \right|\le B^k\sup_{1\le s\le k}\left|\int_0^{\infty}e^{-sv/\lambda}dg(v,\kappa)\right| + \frac{C}{k^{\gamma}}|\partial\Omega|\lambda^{(m+1)/2}.
\end{equation}
Note we have absorbed $\frac 12L^{sc}_{1,m-1}$ into $C$ here. Now we apply \eqref{eq:hkremainderineq} with $t=s/\lambda$ and see immediately that, possibly renaming some constants,
\begin{equation}\label{eq:tauberianresultmk2}
\left|\int_0^{\lambda}(1-v/\lambda)^{\gamma-1}dg(v,\kappa) \right|\le B^kC_0\lambda^{m/2} + \frac{C}{k^{\gamma}}|\partial\Omega|\lambda^{(m+1)/2}.
\end{equation}

As in \cite{franklarson} we now make an optimal choice of $k$. In this case we let
\[k=\max\{k_0,\lfloor \frac 12\log_B(\sqrt{\lambda})\rfloor\}.\]
Then for sufficiently large $\lambda$, specifically
\[\lambda\ge B^{4k_0},\]
we have
\begin{equation}\label{eq:tauberianresultmk3}
\left|\int_0^{\lambda}(1-v/\lambda)^{\gamma-1}dg(v,\kappa) \right|\le C_0\lambda^{m/2 + 1/4} + \frac{C'}{(\log|\lambda|)^{\gamma}}|\partial\Omega|\lambda^{(m+1)/2}.
\end{equation}
Putting this together, for this domain, there exist constants $C$ and $\Lambda$ both depending on the domain such that for $\lambda\ge\Lambda$,
\begin{equation}\label{eq:tauberianresultmk4}
\left|\int_0^{\lambda}(1-v/\lambda)^{\gamma-1}dg(v,\kappa) \right|\le \frac{C}{(\log|\lambda|)^{\gamma}}\lambda^{(m+1)/2},
\end{equation}
which in particular is $o(\lambda^{(m+1)/2})$.

On the other hand, we may compute directly that
\begin{multline}\label{eq:gcomputation}
\int_0^{\lambda}(1-v/\lambda)^{\gamma-1}dg(v,\kappa) = \int_0^{\lambda}(1-v/\lambda)^{\gamma-1}d\Tr(\Delta^0-v)_- - \int_0^{\lambda}(1-v/\lambda)^{\gamma-1}d\Tr(\Delta^\kappa-v)_-\\ - \int_0^{\lambda}(1-v/\lambda)^{\gamma-1}dh(v,\kappa).
\end{multline}
The first term on the right hand side of \eqref{eq:gcomputation} may be integrated by parts. Both boundary terms vanish, yielding
\[\lambda^{-1}\int_0^{\lambda}(\gamma-1)(1-v/\lambda)^{\gamma-2}\Tr(\Delta^0-v)_-\, dv.\]
Changing variables to $\tau=\lambda-v$ gives
\[\lambda^{-\gamma+1}(\gamma-1)\int_0^{\infty}\tau^{\gamma-2}\Tr(\Delta_0-\lambda+\tau)_- d\tau.\]
Applying \cite[(19)]{franklarson} yields
\[\frac{\lambda^{1-\gamma}}{\gamma}\Tr (\Delta^0-\lambda)^{\gamma}_-.\]
Similarly, the second term on the right-hand side of \eqref{eq:gcomputation} is
\[-\frac{\lambda^{1-\gamma}}{\gamma}\Tr (\Delta^{\kappa}-\lambda)^{\gamma}_-.\]
Thus
\begin{equation}
\Tr (\Delta^0-\lambda)^{\gamma}_- - \Tr (\Delta^\kappa-\lambda)^{\gamma}_- = \gamma\int_0^{\lambda}(\lambda - v)^{\gamma-1}dh(v,\kappa) + O(\frac{\gamma\lambda^{\gamma+(m-1)/2}}{(\log|\lambda|)^{\gamma}}).
\end{equation}

All that remains is to identify the term with a $dh$. First the first term in \eqref{eq:dh}; plug it in and get
\[\gamma\frac{m+1}{4}L^{sc}_{1,m-1}|\partial\Omega|\int_0^{\lambda}(\lambda-v)^{\gamma-1}v^{(m-1)/2}\, dv.\]
Let $v=\lambda\bar v$ to write this in terms of the beta function:
\[\gamma\frac{m+1}{4}L^{sc}_{1,m-1}|\partial\Omega|\lambda^{\gamma+(m-1)/2}B(\gamma,(m+1)/2).\]
Plugging in the expressions for $L^{sc}_{1,m-1}$ and the beta function gives
\[\frac 14\gamma(m+1)|\partial\Omega|\lambda^{\gamma+(m-1)/2}\frac{1}{(4\pi)^{(m-1)/2}\Gamma(1+(m+1)/2)}\frac{\Gamma((m+1)/2)\Gamma(\gamma)}{\Gamma(\gamma+(m+1)/2)},\]
from which simplification yields
\[\frac 12 L^{sc}_{\gamma,m-1}|\partial\Omega|\lambda^{\gamma+(m-1)/2}.\]

Now the second term in \eqref{eq:dh}. Plug it in and get
\[- \frac{\gamma(m^2-1)}{4\pi}L^{sc}_{1,m-1}|\partial\Omega|\int_0^{\lambda}\int_0^{v}(\lambda-v)^{\gamma-1}(v-u)^{\frac{m-3}{2}}\arctan(\frac{\sqrt u}{\kappa})\, du\, dv.\]
Change the order of integration:
\[-\frac{\gamma (m^2-1)}{4\pi}L^{sc}_{1,m}|\partial\Omega| \int_0^{\lambda}\arctan(\frac{\sqrt u}{\kappa})\int_u^{\lambda}(\lambda-v)^{\gamma-1}(v- u)^{(m-3)/2}\, dv\, du.\]
Let $\bar v=v-u$:
\[-\frac{\gamma (m^2-1)}{4\pi}L^{sc}_{1,m-1}|\partial\Omega| \int_0^{\lambda}\arctan(\frac{\sqrt u}{\kappa})\int_0^{\lambda-u}(\lambda-u-\bar v)^{\gamma-1}(\bar v)^{(m-3)/2}\, d\bar v\, du.\]
Now let $\hat v=\bar v/(\lambda - u)$:
\[-\frac{\gamma (m^2-1)}{4\pi}L^{sc}_{1,m-1}|\partial\Omega| \int_0^{\lambda}\arctan(\frac{\sqrt u}{\kappa})(\lambda -u)^{\gamma+(m-3)/2}\int_0^{1}(1-\hat v)^{\gamma-1}(\hat v)^{(m-3)/2}\, d\hat v\, du.\]
The inner integral is a beta function $B(\gamma,(m-1)/2)$ and thus we get
\[-\frac{\gamma (m^2-1)}{4\pi}L^{sc}_{1,m-1}B(\gamma,(m-1)/2)|\partial\Omega| \int_0^{\lambda}\arctan(\frac{\sqrt u}{\kappa})(\lambda -u)^{\gamma+(m-3)/2}\, du.\]
Writing out the coefficient and the beta function gives significant cancellation, yielding a term
\[-\frac{1}{\pi}(\gamma+\frac{m-1}{2})L^{sc}_{\gamma,m-1}|\partial\Omega|\int_0^{\lambda}\arctan(\frac{\sqrt u}{\kappa})(\lambda -u)^{\gamma+(m-3)/2}\, du.\]
Letting $u=\lambda x^2$:
\[-\frac{2}{\pi}(\gamma+\frac{m-1}{2})L^{sc}_{\gamma,m-1}|\partial\Omega|\lambda^{\gamma+(m-1)/2}\int_0^{1}x\arctan(\frac{\sqrt\lambda}{\kappa}x)(1-x^2)^{\gamma+(m-3)/2}\, dx.\]

Putting everything together, we now obtain the final result:
\begin{multline}
\Tr (\Delta^0-\lambda)^{\gamma}_- - \Tr (\Delta^\kappa-\lambda)^{\gamma}_-
= \frac 12L^{sc}_{\gamma,m-1}|\partial\Omega|\lambda^{\gamma+(m-1)/2}\\ 
-\frac{2}{\pi}(\gamma+\frac{m-1}{2})L^{sc}_{\gamma,m-1}|\partial\Omega|\lambda^{\gamma+(m-1)/2}\int_0^{1}x\arctan(\frac{\sqrt\lambda}{\kappa}x)(1-x^2)^{\gamma+(m-3)/2}\, dx
\\+ O(\frac{\gamma\lambda^{\gamma+(m-1)/2}}{(\log|\lambda|)^{\gamma}}),
\end{multline}
which proves Theorem \ref{thm:riesz}.

\appendix

\section{Proof of Proposition \ref{prop:everythingispc}}

In this appendix we give the proof of Proposition \ref{prop:everythingispc}. We specialize to the case $m=2$ for notational simplicity, but the proof works essentially verbatim in arbitrary dimension.

Consider local coordinates on $\hat M_{\epsilon}^2$. There are six projective coordinate patches that cover it. $T$ is a coordinate in all of them and is always a bdf for tf. $(x,x')$ are also coordinates in all of them, and we let $\tilde x=x-x'$. Without loss of generality we assume that $y'\lesssim y$, as the analysis is identical in the other case, so we need only consider three coordinate patches:
\begin{enumerate}
\item When $y'\lesssim y\lesssim\epsilon$ and $y'\lesssim\epsilon$, use $(\tilde y':=\frac{y'}{y},\hat y:=\frac{y}{\epsilon},\epsilon)$ as bdfs for lf, hnf, and hff respectively.
\item When $y'\lesssim\epsilon\lesssim y$, use $(\frac{y'}{\epsilon},\frac{\epsilon}{y},y)$ as bdfs for lf, hlf, and hff respectively.
\item When $\epsilon\lesssim y'\lesssim y$, use $(\frac{\epsilon}{y'},\frac{y'}{y},y)$ as bdfs for zf, hlf, and hff respectively.
\end{enumerate}
We will show that $H_{corr}$ is pc in each of the three coordinate patches. We begin with patches (2) and (3), then use a different argument for patch (1).

\subsection{Patches (2) and (3)} The critical fact about patches (2) and (3) is that they are away from hnf and thus away from $P_3$. So in these patches,
\[M_{\epsilon,rh}^2 = [\hat M_{\epsilon}^2;P_1].\]
To take advantage of this, we write $H_{corr}$ as a product
\begin{multline}
H_{corr}=\epsilon^{-1}H_{corr,1}\cdot H_{corr,2}\\ := \epsilon^{-1}\left(-\frac{1}{\sqrt{4\pi} T}\exp[-\frac{|x-x'|^2}{4T^2}-\frac{(y+y')^2}{8T^2}]\right)
\cdot\left( \exp[-\frac{(y+y')^2}{8T^2}]v(\frac{y+y'}{2T}+\frac{T}{\epsilon})\right).
\end{multline}

We first claim that $H_{corr,1}$ is pc on $M^2_{rh}$. Indeed it is essentially the half-space correction term from \cite{NRS25}, which is pc on $M^2_{rh}$. One could see this directly from the expression, but it is easiest to observe that replacing $y$ and $y'$ with $\sqrt 2 y$ and $\sqrt 2 y'$ respectively yields precisely the half-space correction term from \cite{NRS25}. Since this replacement leaves the manifold-with-corners structure unchanged (it is a b-map from $M^2_{rh}$ to itself), we have the claim. By Lemma \ref{lem:addtheeps}, we conclude that $\epsilon^{-1}H_{corr,1}$ is pc on $M^2_{rh,\epsilon}$.

The term $H_{corr,2}$ has no $x$-coordinates and thus is not actually pc on $M^2_{rh,\epsilon}$. However:

\begin{lemma}\label{lem:hcorr2} In patches (2) and (3), the expression $H_{corr,2}$ is pc on $[\hat M_{\epsilon}^2;P_5]$.
\end{lemma}
\begin{proof}
In \textbf{patch (2)}, our expression for $H_{corr,2}$ becomes
\[H_{corr,2}=\exp[-\frac{y^2(1+(y'/\epsilon)(\epsilon/y))^2}{8T^2}]v(\frac{y}{2T}(1+(y'/\epsilon)(\epsilon/y))+\frac{T}{y(\epsilon/y)}).\]
The submanifold $P_5$ is $\{T=y=0\}$. Blowing it up creates two new projective coordinate patches. First, if $T\lesssim y$, then we must use $T/y$ instead of $T$ and we obtain
\[H_{corr,2}=\exp[-\frac{(1+(y'/\epsilon)(\epsilon/y))^2}{8(T/y)^2}]v(\frac{1}{2(T/y)}(1+(y'/\epsilon)(\epsilon/y))+(T/y)\frac{1}{(\epsilon/y)}).\]
The exponential term is pc and decays to infinite order as $(T/y)\to 0$. On the other hand, the argument of $v$ is bounded away from zero and its inverse, the pullback of the boundary defining function $x^{-1}$ at $x=\infty$ in the codomain, is
\[2(T/y)\cdot\frac{1}{1+(y'/\epsilon)(\epsilon/y)+\frac{(T/y)}{(\epsilon/y)}}.\]
If this were a product of bdfs times a smooth nonvanishing function, then $v$ would be a b-map and we would be done by the pullback theorem. It is not.

Nevertheless, we are fine, as the exponential decay of the first term at the problematic corner comes to the rescue. Specifically, we blow up $P_7:=\{(T/y)=(\epsilon/y)=0\}$. After creating this new front face, the pullback of $x^{-1}$ does have the desired form, so the $v$ term is pc on $[\hat M_{\epsilon}^2;P_5;P_7]$. But the exponential term is pc on this space as well \emph{and} decays to infinite order at the new front face. Thus $H_{corr,2}$ decays to infinite order at the new front face, and so was already pc on $[\hat M_{\epsilon}^2;P_5]$.

We still need to consider the other patch, with $y\lesssim T$, where we use $y/T$ instead of $y$ and obtain
\[H_{corr,2}=\exp[-\frac 18(y/T)^2(1+(y'/\epsilon)(\epsilon/y))^2]v(\frac 12 (y/T)(1+(y'/\epsilon)(\epsilon/y))+\frac{1}{(y/T)(\epsilon/y)}).\]
The inverse of the argument of $v$ is now
\[2(\epsilon/y)(y/T)\frac{1}{2+(y/T)^2(\epsilon/y)(1+(y'/\epsilon)(\epsilon/y))},\]
which is a product of bdfs times a smooth nonvanishing function. The pullback theorem thus finishes the job and completes the analysis of patch (2).

\textbf{Patch (3)} is similar; our expression for $H_{corr,2}$ is
\[H_{corr,2}= \exp[-\frac{y^2(1+(y'/y))^2}{8T^2}]v(\frac{y}{2T}(1+(y'/y))+\frac{T}{(\epsilon/y')(y'/y)y}).\]
Again $P_5$ is $\{T=y=0\}$ and blowing it up creates two patches. If $T\lesssim y$ then the exponential term is decaying to infinite order at $(T/y)=0$, and the inverse of the argument of $v$ is
\[\frac{(T/y)(\epsilon/y')(y'/y)}{(\epsilon/y')(y'/y)(1+(y'/y)^2) + (T/y)^2}.\]
To see the structure, let $w_1=(\epsilon/y)(1+(y'/y))$, $w_2=(y'/y)$, $w_3=(T/y)$. Then this inverse argument is $w_1w_2w_3$ times a smooth nonvanishing function of them times
\[\frac{1}{w_1w_2+w_3^2}.\]
This singularity can be resolved by doing the following blow-ups: first $\{w_1=w_2=w_3=0\}$, then $\{w_2=w_3=0\}$ and $\{w_1=w_3=0\}$. With these the inverse argument has the desired form. To see this, the blown-up space has five coordinate patches, two of which are isomorphic by switching the indices 1 and 2.
\begin{itemize}
\item In the patch $(w_1/w_3, w_2/w_3,w_3)$ we have
\[w_3^{-2}\frac{1}{1+(w_1/w_3)(w_2/w_3)}.\]
\item In the patch $(w_1/w_3,w_3/w_2,w_2)$ we have
\[w_2^{-2}(w_3/w_2)^{-1}\frac{1}{1+(w_1/w_3)}.\]
\item In the patch $(w_3/w_1,w_1/w_2,w_2)$ we have
\[w_2^{-2}(w_1/w_2)^{-1}\frac{1}{1+(w_3/w_1)^2(w_1/w_2)}.\]
\end{itemize}
In each patch, the fraction is smooth and nonvanishing. So after these three additional blow-ups, the $v$ term is pc. However, each additional blow-up was done at $(T/y)=0$, where the exponential term has infinite order decay. By the same argument as before, $H_{corr,2}$ is pc here.

Finally, if $y\lesssim T$ then we get a smooth exponential term and the inverse of the argument of $v$ is
\[(\epsilon/y')(y'/y)(y/T)\frac{1}{1+(\epsilon/y')(y'/y)(y/T)^2(1+(y'/y))}.\]
This again shows that $v$ is a b-map. Thus $H_{corr,2}$ is pc on $[\hat M_{\epsilon}^2;P_5]$ in patch (3), completing the proof of Lemma \ref{lem:hcorr2}.
\end{proof}
To complete the analysis, observe that since $P_1\subseteq P_5$,
\[ [\hat M_{\epsilon}^2;P_1;P_5] = [\hat M_{\epsilon}^2;P_5;P_1],\]
so this space is a blow-up of both $[\hat M_{\epsilon}^2;P_1]$ and $[\hat M_{\epsilon}^2;P_5]$. Therefore both factors $H_{corr,1}$ and $H_{corr,2}$ are pc here, and thus so is their product $H_{corr}$, But $H_{corr,1}$ decays to infinite order at $P_5$, so $H_{corr}$ does as well, and thus $H_{corr}$ was in fact already pc on $[\hat M_{\epsilon}^2;P_1]$. In patches (2) and (3), that shows that $H_{corr}$ is pc on $M_{\epsilon,rh}^2$.

\subsection{Patch (1)} In this patch we must consider $H_{corr}$ all together as opposed to separately. We have the coordinates $(\epsilon, \hat y, \tilde y', T, \tilde x, x)$. In these coordinates,
\[H_{corr}=\exp[-\frac{\tilde x^2}{4T^2}-\frac{\epsilon^2\hat y^2(1+\tilde y')^2}{4T^2}]v(\frac{\epsilon \hat y(1+\tilde y')}{2T}+\frac{T}{\epsilon}).\]
We blow up $P_1$ to create sf$_0$ and then $P_3$ to create sf, and in these coordinates
\[P_1 = \{\epsilon=T=\tilde x=0\};\quad P_3=\{\hat y=T=\tilde x=0\}.\]
Observe that the $x$ coordinate does not appear in $H_{corr}$; the $\tilde y'$ coordinate is also harmless, and the real interest is the interplay between $\epsilon$, $T$, $\hat y$, and $\tilde x$.

When blowing up $P_1$, we introduce three new coordinate sub-patches of patch (1), depending on which of $T$, $\epsilon$, or $\tilde x$ is the largest. (We assume without loss of generality throughout that $\tilde x\ge 0$; the negative case is identical, just use $-\tilde x$ instead).
\begin{enumerate}
\item[i)] If $T$ is the largest, use $(T,\tilde\epsilon:=\frac{\epsilon}{T},u:=\frac{\tilde x}{T})$.
\item[ii)] If $\tilde x$ is the largest, use $(\tilde x,T_0:=\frac{T}{\tilde x},\epsilon_0:=\frac{\epsilon}{\tilde x})$.
\item[iii)] If $\epsilon$ is the largest, use $(\epsilon,\hat T:=\frac{T}{\epsilon},\hat x:=\frac{\tilde x}{\epsilon})$.\\
\end{enumerate}

In \textbf{sub-patch i)}, we have
\[H_{corr}=\exp[-\frac 14 u^2-\frac 14\tilde\epsilon^2\hat y^2(1+\tilde y')^2]v(\frac 12\tilde\epsilon \hat y(1+\tilde y')+\frac{1}{\tilde\epsilon}).\]
The exponential term is just smooth. As in the proofs in patches (2) and (3), the argument of $v$ is bounded away from zero and its inverse is
\[\tilde\epsilon(\frac{1}{1+\frac 12\tilde\epsilon^2\hat y(1+\tilde y')}),\]
which is a product of bdfs times a smooth nonvanishing function. Thus the $v$ term is pc here. So in this sub-patch, $H_{corr}$ is already pc after blowing up just $P_1$. It is therefore pc on $M^2_{rh,\epsilon}$.\\

In \textbf{sub-patch ii)}, we have
\[H_{corr}=\exp[-\frac{1}{4T_0^2}-\frac{\epsilon_0^2\hat y^2(1+\tilde y')^2}{4T_0^2}]v(\frac{\epsilon_0 \hat y(1+\tilde y')}{2T_0}+\frac{T_0}{\epsilon_0}).\]
The exponential term is smooth here and decays to infinite order at $T_0=0$. The $v$ term is not pc -- but we can fix that with blow-ups that are all in the lift of $T_0=0$. Specifically, let $w_1=\epsilon_0$, $w_2=\hat y(1+\tilde y')/2$, $w_3=T_0$, then we have
\[v(\frac{w_1w_2}{w_3}+\frac{w_3}{w_1}).\]
We claim that this is pc after blowing up first $\{w_1=w_3=0\}$ then the lift of $\{w_2=w_3=0\}$; since $T_0=w_3$ is zero at both of these new faces, $H_{corr}$ decays there to infinite order and thus the blow-ups are unnecessary. So this claim is sufficient. 

To prove the claim, if $w_1\lesssim w_3$ then we have 
\[v((w_1/w_3)w_2 + \frac{1}{(w_1/w_3)}).\]
The argument is bounded away from zero and its inverse is
\[(w_1/w_3)\cdot\frac{1}{1+(w_1/w_3)^2w_2},\]
which is a product of bdfs times a smooth nonvanishing function as needed. On the other hand, if $w_3\lesssim w_1$ then we have
\[v(\frac{w_2}{(w_3/w_1)}+(w_3/w_1)),\]
which needs more work. But this is why we blow up the lift of $\{w_2=w_3=0\}$, which in these coordinates is $\{w_2=(w_3/w_1)=0\}$. If $w_2$ is smaller, the argument of $v$ is bounded away from infinity and is smooth; if it is larger, then we have
\[v(\frac{1}{((w_3/w_1)/w_2)}+w_2((w_3/w_1)/w_2).\]
The argument of $v$ is bounded away from zero and its inverse is
\[((w_3/w_1)/w_2)\frac{1}{1+w_2((w_3/w_1)/w_2)^2},\]
which again has the desired form. This completes the analysis of sub-patch ii).\\

In \textbf{sub-patch iii)}, we have
\[H_{corr}=\exp[-\frac{\hat x^2}{4\hat T^2}-\frac{\hat y^2(1+\tilde y')^2}{4\hat T^2}]v(\frac{\hat y(1+\tilde y')}{2\hat T}+\hat T).\]
This is not yet pc; in this patch and only this patch, we also need to blow up $P_3$, which is $\{\hat y=\hat T=\hat x=0\}$. After doing so there are, yet again, three coordinate patches:
\begin{itemize}
\item If $\hat T$ is largest then we have, where everything in parentheses is a projective coordinate,
\[H_{corr}=\exp[-\frac 14(\hat x/\hat T)^2-\frac 14(\hat y/\hat T)^2(1+\tilde y')^2]v(\frac 12(\hat y/\hat T)(1+\tilde y')+\hat T).\]
Here the exponential is smooth and the argument of $v$ is smooth and bounded away from infinity, so the expression is pc.
\item If $\hat y$ is largest then, as with the first case,
\[H_{corr}=\exp[-\frac{(\hat x/\hat y)^2}{4(\hat T/\hat y)^2}-\frac{(1+\tilde y')^2}{4(\hat T/\hat y)^2}]v(\frac{(1+\tilde y')}{2(\hat T/\hat y)}+(\hat T/\hat y)\hat y).\]
The exponential term is smooth and vanishes to infinite order at $(\hat T/\hat y)=0$. The argument of $v$ is bounded away from zero and its inverse is $(\hat T/\hat y)$ times a smooth nonvanishing function. So the expression is pc.
\item If $\hat x$ is largest write $\overline T=\hat T/\hat x$ and $\overline y=\hat y/\hat x$; we get

\[H_{corr}=\exp[-\frac{1}{4\overline T^2}-\frac{\overline y^2(1+\tilde y')^2}{4\overline T^2}]v(\frac{\overline y(1+\tilde y')}{2\overline T}+\overline T\hat x).\]
The exponential term is smooth and vanishes to infinite order at $\overline T=0$. The $v$ term is not yet pc but all that we have to do is blow up $\{\overline y=\overline T=0\}$, where the exponential term vanishes to infinite order. If $\overline T$ is bigger, the argument of $v$ is bounded and smooth; if $\overline y$ is bigger, the argument is bounded away from zero and its inverse is the bdf $\overline T/\overline y$ times a smooth nonvanishing function. In either case, $H_{corr}$ is pc.
\end{itemize}

We have now completed the proof of Proposition \ref{prop:everythingispc}.

\end{document}